\documentclass{statsoc}

\usepackage{amsmath}
\usepackage{amssymb}
\usepackage{graphicx}
\usepackage{verbatim}
\usepackage{natbib}
\usepackage{subfigure}
\usepackage{multirow}
\usepackage{hyperref}
\usepackage{enumerate}

\newcommand{\pp}{\mathcal{P}}
\newcommand{\mm}{\mathcal{M}}

\newcommand{\E}{\mathbb{E}}

\newcommand{\simiid}{\,{\buildrel \text{iid} \over \sim\,}}

\DeclareMathOperator{\Exp}{Exp}

\newtheorem{theo}{Theorem}
\newtheorem{lemma}[theo]{Lemma}
\newtheorem{coro}[theo]{Corollary}

\newtheorem{proc}{Procedure}

\newcommand{\RR}{\mathbb{R}}

\newcommand{\ff}{\mathcal{F}}

\newcommand{\nn}{\mathcal{N}}

\newcommand{\EE}[2][]{\mathbb{E}_{#1}\left[#2\right]}
\newcommand{\PP}[2][]{\mathbb{P}_{#1}\left[#2\right]}

\newcommand{\eqd}{\,{\buildrel d \over =}\,}

\newcommand{\hk}{\hat{k}}
\newcommand{\hl}{\hat{l}}

\newcommand{\hatt}{\hat{t}}

\newcommand{\hbeta}{\hat{\beta}}

\newcommand{\tp}{\tilde{p}}
\newcommand{\tq}{\tilde{q}}

\newcommand{\tY}{\widetilde{Y}}
\newcommand{\tZ}{\widetilde{Z}}

\newcommand{\with}{\text{ with }}
\newcommand{\eqand}{\text{ and }}

\newcommand{\hkF}{\hk_F}
\newcommand{\hkS}{\hk_S}
\newcommand{\hkT}{\hk_T}

\graphicspath{{./figures/}}

\newcommand{\inc}{\text{inc}}
\newcommand{\comp}{\text{comp}}

\newcommand{\FM}{\text{FM}}

\author{Max Grazier G'Sell}
\address{Department of Statistics, Carnegie Mellon University, Pittsburgh, USA.}
\author{Stefan Wager}
\address{Department of Statistics, Stanford University, Stanford, USA.}
\author{Alexandra Chouldechova}
\address{Heinz College, Carnegie Mellon University, Pittsburgh, USA.}
\author{Robert Tibshirani}
\address{Departments of Health Research \& Policy, and Statistics, Stanford University, Stanford, USA.}
\email{mgsell@cmu.edu; swager@stanford.edu}

\title[Sequential FDR Control]{Sequential Selection Procedures and \\ False Discovery Rate Control}
\date{\today}

\begin{document}

\maketitle

\begin{abstract}

We consider a multiple hypothesis testing setting where the hypotheses are
ordered and one is only permitted to reject an initial contiguous block,
$H_1,\,\dots,H_k$, of hypotheses.  A rejection rule in this setting amounts to a
procedure for choosing the stopping point $k$.  This setting is inspired by the
sequential nature of many model selection problems, where choosing a stopping
point or a model is equivalent to rejecting all hypotheses up to that point and
none thereafter.  We propose two new testing procedures, and prove that they
control the false discovery rate in the ordered testing setting.  We also show
how the methods can be applied to model selection using recent results on
$p$-values in sequential model selection settings.  

\end{abstract}

\keywords{multiple hypothesis testing, stopping rule, false discovery rate, sequential testing}

\section{Introduction}\label{section:intro}

Suppose that we have a sequence of null hypotheses, $H_1$, $H_2, \ldots H_m$, and that
we want to to reject some hypotheses while controlling the False
Discovery Rate \citep[FDR,][]{benjamini1995controlling}.
Moreover, suppose that these hypotheses
must be rejected in an ordered fashion: a test procedure must reject
hypotheses $H_1, \, \dots, \, H_k$ for some $k \in \{0,1, \,\dots, \,m\}$.
Classical methods for FDR control, such as the original Benjamini-Hochberg selection procedure,
are ruled out by the requirement that the hypotheses be rejected in order.

In this paper we introduce new testing procedures that address this problem, and control the False
Discovery Rate (FDR) in the ordered setting.
Suppose that we have a sequence of $p$-values, $p_1$, ..., $p_m \in [0,1]$ corresponding
to the hypotheses $H_j$, such that $p_j$ is uniformly distributed on $[0,1]$ when $H_j$ is true.
Our proposed methods start by transforming
the sequence of $p$-values $p_1$, ..., $p_m$ into a monotone increasing sequence of
statistics $0 \leq q_1 \leq \ldots \leq q_m \leq 1$. We then prove that we achieve
ordered FDR control by applying the original Benjamini-Hochberg procedure on the
monotone test statistics $q_i$. 

\subsection{Variable Selection along a Regression Path}
\label{sec:path}

This problem of FDR control for ordered hypotheses arises naturally when
implementing variable selection using a path-based a path-based regression algorithm; examples of such algorithms include
forward stepwise regression \citep[see][for a review]{hocking1976biometrics}
and least-angle regression \citep{LARS}. These methods 
build models by adding in variables one-by-one, and the number of
non-zero variables in the final model only depends on a single
sparsity-controlling tuning parameter.
The lasso \citep{lasso} can also be used for
path-based variable selection; however, the lasso also sometimes
removes variables from its active set while building its model.

Each time we add a new variable to the model, we may want to
ask---heuristically---whether adding the new variable to the model
is a ``good idea''. Because the path algorithm specifies the order in which
variables must be added to the model, asking these questions yields
a sequence of ordered hypotheses for which it is desirable to control the
overall FDR.

To fix notation, suppose that we have data $X \in \RR^{n \times p}$ and
$Y \in \RR^n$, and seek to fit the linear regression model
$$ Y \sim \nn\left(X \beta ^*, \, \sigma^2 I_{p \times p}\right) $$
using a sparse weight vector $\hbeta$. Path algorithms can then be
seen as providing us with an ordering of the variables
$ j_1, \, j_2, \, ... \in \{1,\, ..., \, p\} $
along with a sequence of nested models
$$ \emptyset = \mm_0 \subset \mm_1 \subset ... \subset \mm_p, \; \text{with} \; \mm_k = \{j_1, \, ..., \, j_k\}. $$
The statistician then needs to pick one of the models $\mm_k$, and
set to zero all coordinates $\hbeta_j$ with $j \notin \mm_k$.
The $k$-th ordered hypothesis $H_k$ tests whether
or not adding the $k$-th variable $j_k$ was informative. 

The null hypothesis $H_k$ that adding the $k$-th variable
along the regression path was uninformative can be formalized
in several ways.
\begin{itemize}
\item {\bf The Incremental Null:} In the spirit of the classical
AIC \citep{akaike1974new} and BIC \citep{schwarz1978estimating} procedures,
$H_k$ measures whether model $\mm_k$ improves over $\mm_{k-1}$.
In the case of linear regression, the null hypothesis states that the best regression
fit for model $\mm_{k-1}$ is the same as the best regression fit for $\mm_k$ or,
more formally:
\begin{align}
\label{eq:incremental}
&H_k^{\inc}: \pp_{\mm_{k - 1}} X \beta^* = 
\pp_{\mm_{k}} X \beta^*, \; \text{where} \\
&\pp_{\mm} = X_{\mm} \left(X_{\mm}^\top X_{\mm}\right)^{\dagger} X_{\mm}
\end{align}
is a projection onto the column-span of $X_\mm$. Here, we
write $X_{\mm}$ for the matrix comprised of the columns of $X$ contained
in $\mm$, and $A^\dagger$ denotes the Moore-Penrose pseudoinverse of 
a matrix A. \citet{taylor2014post} develop tests for $H_k^\inc$ in the context of
both forward stepwise regression and least-angle regression.
\item {\bf The Complete Null:} We may also want to test the stronger null hypothesis
that the model $\mm_{k - 1}$ already captures all the available signal. More specifically,
writing $\mm^*$ for the support set of $\beta^*$, we define
\begin{equation}
\label{eq:complete}
H_k^\comp: \mm^* \subseteq \mm_{k - 1}.
\end{equation} 
Tests of $H_k^\comp$ for various pathwise regression models
have been studied by, among others, \citet{lockhart2013significance},
\citet{fithian2014optimal}, \citet{loftus2014significance}, and \citet{taylor2013tests}.
\item {\bf The Full-Model Null:} Perhaps the simplest pathwise hypothesis we may want
to test is that
\begin{equation}
\label{eq:abs}
H_k^\FM: j_k \in \mm^*,
\end{equation}
i.e., that the $k$-th variable added to the regression path belongs to the support set of
$\beta^*$. Despite its simple appearance, however, the hypothesis $H_k^\FM$ is
difficult to work with. The problem is that the truth of $H_k^\FM$ depends critically on variables
that may not be contained in $\mm_k$, and so $H_k^\FM$ will have a ``high-dimensional''
character even when $k$ is small. We are not aware of any general methods for testing
$H_k^\FM$ along, for example, the least-angle regression path, and do not pursue this
formalization further in this paper.
\end{itemize}
The incremental and complete null hypotheses may both be appropriate in different
contexts, depending on the needs of the statistician. An advantage of testing $H_k^\inc$
is that it seeks parsimonious models where most non-zero variables are useful. On the
other hand, $H_k^\comp$ has the advantage that, unlike with $H_k^\inc$, subsequent
hypotheses are nested; this can make interpretation easier. We note that, when $X$ has
full column rank: 
\begin{align*}
  H_k^\comp = \bigwedge_{l = k}^p H_k^\inc.
\end{align*}

The goal of this paper is to develop generic FDR control procedures for ordered hypotheses,
that can be used for pathwise variable selection regardless of a statistician's
choice of fitting procedure (forward stepwise or least-angle regression),
null hypothesis ($H_k^\inc$ or $H_k^\comp$), and test statistic.
The flexibility of our approach should be a major asset, as the proliferation
of methods for pathwise hypothesis testing suggests an interest in the topic
\citep{lockhart2013significance,loftus2014significance,fithian2014optimal,
gsell2013adaptive,lee2013exact,lee2014exact,taylor2013tests,taylor2014post}.

\begin{table}
\caption{\label{tab:ex1}Typical realization of $p$-values for $H_k^\inc$ with least-angle regression (LARS), as
proposed by \citet{taylor2014post}.  }
\centering
\begin{tabular}{|l|rrrrrrrrrr|}
  \hline
LARS step&1&2&3&4&5&6&7&8&9&10\\
\hline
Predictor&3 &1 &4 &10 &9 &8 &5 &2 &6 &7\\
$p$-value&0.00 & 0.08 &0.34 &0.15& 0.93& 0.12& 0.64& 0.25& 0.49 & $\cdot$\\
\hline
\end{tabular}
\end{table}

{\bf Example.} To further illustrate our setup, consider a simple model selection problem.
We have $n$ observations from a  linear model with $p$ predictors,
\begin{equation}
\label{eqn:linreg}
y_i=\beta_0 +\sum_{j=1}^p x_{ij}\beta_j+  Z_i \with Z_i \sim \nn(0,1),
\end{equation}
and seek to fit $\beta$ by least-angle regression. As discussed above,
this procedure adds variables to the model one-by-one, and we need to
decide after how many variables $k$ to stop. The recent work of
\citet{taylor2014post} provides us with $p$-values for the sequence of hypotheses
$H_k^\inc$ defined in \eqref{eq:incremental}; Table \ref{tab:ex1} has a typical realization of
these $p$-values with data generated from a model
$$ n=50, \, p=10, \, x_{ij} \simiid \nn(0,1), \, \beta_1=2, \, \beta_3=4,  \, \beta_2 = \beta_4=\beta_5 \ldots \beta_{10}=0. $$
These $p$-values are not exchangeable, and must be treated 
in the order in which the predictors were entered: 3, 1, 4 etc.
Our goal is to use these $p$-values to produce an FDR-controlling stopping
rule.  In the following section, we introduce two procedures:
\emph{ForwardStop} and \emph{StrongStop} that control FDR.  
Figure \ref{fig:robexample} illustrates the performance of one of our proposed procedures; in this
example, it allows us to accurately estimate the support of $\beta$ while
successfully controlling the FDR.

\begin{figure}[t]
  \centering
  \includegraphics[width=0.9\textwidth]{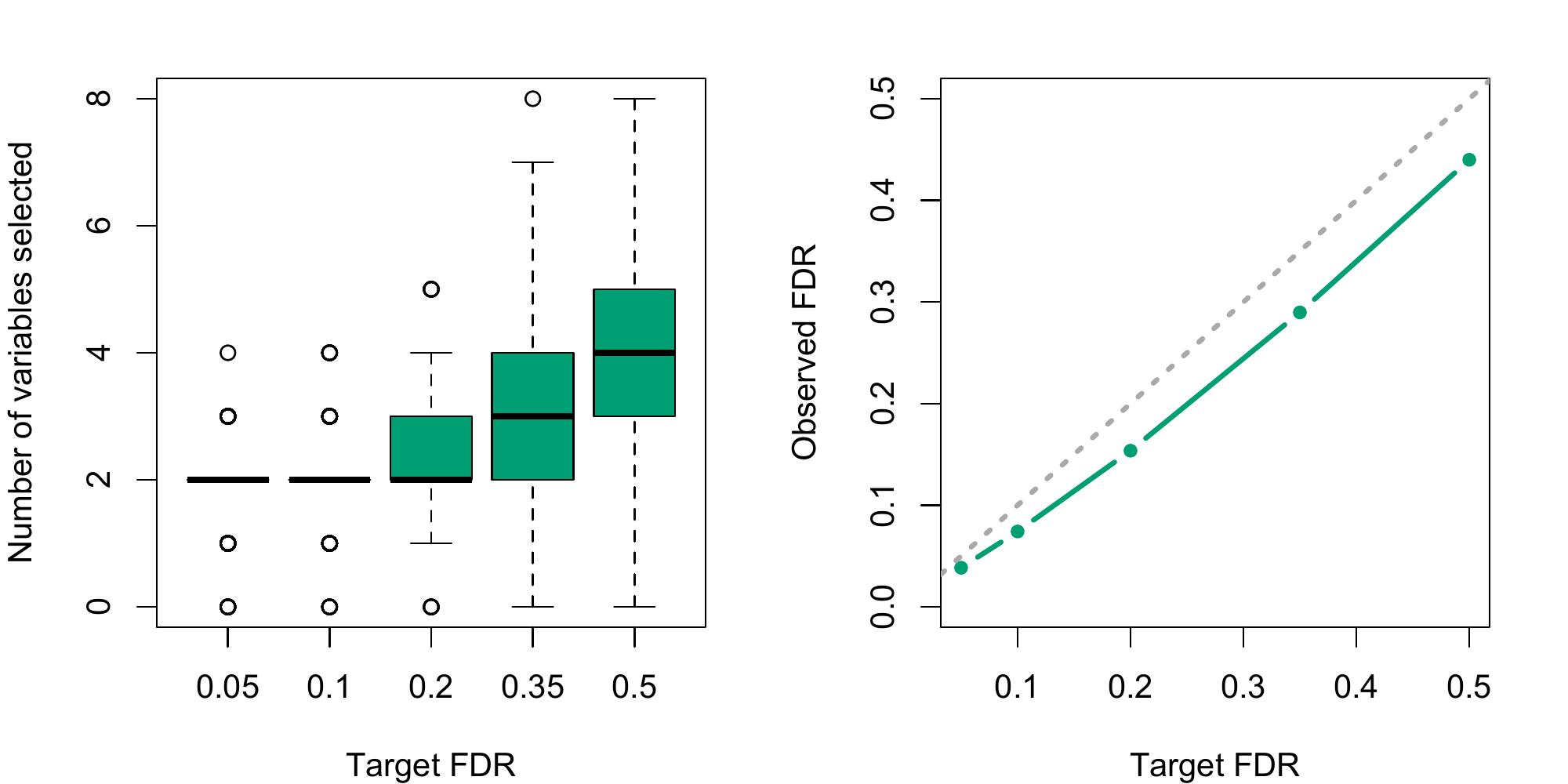}
  \caption{For the model selection problem in
Equation \eqref{eqn:linreg}, 1000 random realizations were simulated and the \emph{ForwardStop} procedure
applied.  The left panel shows the number of predictors selected at FDR levels 0.05,
0.1, 0.2, 0.35 and 0.5.  The right panel shows the observed FDR on the Y axis and
the Target FDR on the X axis.  The $45^o$ line is plotted in grey for reference.}
  \label{fig:robexample}
\end{figure}

\subsection{Stopping Rules for Ordered FDR Control}\label{section:intro:methodology}

In the ordered setting, a valid
rejection rule is a function of $p_1$ , ..., $p_m$ that returns a cutoff
$\hat{k}$ such that hypotheses $H_1,\dots,H_{\hat{k}}$ are rejected.  The
False Discovery Rate (FDR) is defined as 
$\E\left[{V(\hat{k})} / {\max(1,\hat{k})}\right]$, where $V(\hat{k})$ is the
number of null hypotheses among the rejected hypotheses $H_1$, ... , $H_{\hat{k}}$.  

We propose two rejection functions for this scenario, called \emph{ForwardStop}:
\begin{equation}
\hkF = \max\left\{k\in\{1,\, \dots, \,m\}: -\frac{1}{k}\sum_{i=1}^k
\log(1-p_i) \leq \alpha \right\},
\label{eq:ForwardStop}
\end{equation}
 and \emph{StrongStop}:
 \begin{equation}
\hkS = \max\left\{k\in\{1, \,\dots, \,m\}: \exp\left(\sum_{j=k}^m \frac{\log p_j}{j}\right) \leq \frac{\alpha k}{m}\right\}.
\label{eq:StrongStop}
\end{equation}
We adopt the convention that $\max(\emptyset) = 0$, so that $\hat k = 0$ whenever no rejections can be made.
In Section \ref{section:seqtheory} we show that both \emph{ForwardStop} and \emph{StrongStop} control FDR at level $\alpha$.

\emph{ForwardStop}  first transforms the $p$-values, and then sets the rejection threshold at the
largest $k$ for which the first $k$ transformed $p$-values have a small enough
average. If the first $p$-values are very small, then \emph{ForwardStop} will always
reject the first hypotheses regardless of the last $p$-values.  As a result, the rule is moderately robust to potential misspecification of the null distribution of the $p$-values at high indexes.  This is particularly important in model selection applications, where one may doubt whether the asymptotic distribution is accurate in finite samples at high indexes.

Our second rule, \emph{StrongStop} \eqref{eq:StrongStop}, comes with a stronger
guarantee than \emph{ForwardStop}. As we show in Section
\ref{section:seqtheory}, provided that the non-null $p$-values precede the null
ones, it not only controls the FDR, but also controls the  Family-Wise Error
Rate (FWER) at level $\alpha$.  Recall that the FWER is the probability that a
decision rule makes even a single false discovery. If false discoveries have a
particularly high cost, then \emph{StrongStop} may be more attractive than
\emph{ForwardStop}. The main weakness of \emph{StrongStop} is that the decision
to reject at $k$ depends on all the $p$-values after $k$.  If the very last
$p$-values are slightly larger than they should be under the uniform
hypothesis, then the rule suffers a considerable loss of power.

A major advantage of both \emph{ForwardStop} and \emph{StrongStop} is that
these procedures seek the largest $k$ at which an inequality holds, 
even if the inequality may not hold for some index $l$ with $l < k$. This property enables them
to get past some isolated large $p$-values for the early hypotheses, thus resulting
in a substantial increase in power. This phenomenon is closely related to the gain
in power of the \citet{benjamini1995controlling} procedure over
the \citet{simes1986improved} procedure.

\subsection{Related Work}

Although there is an extensive literature on FDR control and its variants
\citep[e.g.,][]{benjamini1995controlling,BY01,blanchard2008two,ETGC01,goeman2010sequential,romano2006stepup,storey2004strong}, no
definitive procedure for ordered FDR control has been proposed so far. The closest
method we are aware of is an adaptation of the $\alpha$-investing approach
\citep{aharoni2013generalized,foster2008alpha}. However, this procedure is not known to formally control
the FDR (\citeauthor{foster2008alpha} prove that it controls the mFDR, defined as
$\E V / (\E R + \eta)$ for some constant $\eta$); moreover, in our simulations, this
approach has lower power than our proposed methods.

The problem of providing FDR control in regression models has been studied, among others,
by \citet{barber2014controlling},  \citet{benjaminigavrilov2009}, \citet{bogdan2013statistical}, \citet{linfoster2011vif},
\citet{meinshausen2010stability}, \citet{samworth2012stability}, and \cite{wuboos2007pseudovar},
using a wide variety of ideas involving resampling, pseudo-variables, and specifically tailored
selection penalties. The goal of our paper is not to directly compete with these methods,
but rather to provide ``theoretical glue'' that lets us transform the rapidly growing family of
sequential $p$-values described in Section \ref{sec:path}
into model selection procedures with FDR guarantees.

We note that the problem of variable selection for regression models can be thought of as a generalization of the standard multiple testing problems, where each $p$-value corresponds to its own variable \citep[e.g.,][]{churchill1994empirical,10002012integrated,simonsen2004using,westfall1993resampling}.  In a standard genome-wide association study, for instance, one might test a family of hypotheses of the form $H_{i,0}: \text{SNP } i$ is associated with the response, $i = 1, \ldots, m$.  When there is high spatial correlation across SNP's, the set of rejected hypotheses is likely to contain correlated subgroups of SNP's that are redundant: while each is marginally significant, all SNP's in a subgroup carry essentially the same information about the response.  The goal of model selection is to avoid this type of redundancy by selecting a group of SNP's each of which contains significant distinct information about the response.

It is also important to contrast the goal of our work with that of prediction-driven model selection procedures such as cross-validation.  Prediction-driven approaches select models that minimize the estimated prediction error, but generally provide no guarantee on the statistical significance of the selected predictors.   Our goal is to conduct inference to select a parsimonious model with inferential guarantees, even though the selected model will generally be smaller than the model giving the lowest prediction error.  

Finally, a key challenge in conducting inference in regression settings is dealing with correlated predictors.
Indeed, when the predictors are highly
correlated, the appropriateness (and definition) of FWER and FDR as error criteria
may come into question.  If we select a noise variable that is highly correlated with a signal
variable, should we consider it to be a false selection?  This is a broad
question that is beyond the scope of this paper, but is worth considering when
discussing selection errors in problems with highly correlated $X$.  This
question is discussed in more detail in several papers
\citep[e.g.,][]{benjaminigavrilov2009,bogdan2013statistical,gsell2013fvr,linfoster2011vif,wuboos2007pseudovar}. 

\subsection{Outline of this paper}

We begin by presenting generic methods for FDR control in ordered settings.
Section \ref{section:seqtheory} develops our two main proposals for sequential
testing, \emph{ForwardStop} and \emph{StrongStop}, along with their theoretical
justification.  We evaluate these rules on simulations in Section
\ref{section:examples}.  In Section \ref{section:application}, we review the
recent literature on sequential testing for model selection
problems and discuss its relation to our procedures. Moreover, we
develop a more specialized version of \emph{StrongStop}, called
\emph{TailStop}, which takes advantage of special properties of some of the
proposed sequential tests.  Finally, in \ref{section:appsimulation}, we evaluate
our sequential FDR controlling procedures in combination with pathwise regression test statistics of
\citet{lockhart2013significance} and \citet{taylor2014post} in both simulations
and a real data example.

All proofs are provided in Appendix \ref{section:appendixproof}.

\section{False Discovery Rate Control for Ordered Hypotheses}
\label{section:seqtheory}

In this section, we study a generic ordered layout where we test a sequence of hypotheses
that are associated with $p$-values $p_1, \, ..., \, p_m \in [0, 1]$.
A subset $N \subset \{0, \, ..., \, m\}$ of these $p$-values are null, with the property that
\begin{equation}
\label{eq:setup}
\{p_i : i \in N\} \simiid U([0, 1]).
\end{equation}
We can reject the $k$ first hypotheses for some $k$ of our choice. Our goal is to make $k$ as large as possible, while controlling the number of false discoveries
$$ V(k) = \left|\{i \in N : i \leq k\}\right|. $$
Specifically, we want to use a rule $\hk$ with a bounded false discovery rate
\begin{equation}
\label{eq:FDR_def}
{\rm FDR}(\hk) = \EE{{V(\hk)} \,\Big/\, {\max\left\{\hk, \, 1\right\}}}.
\end{equation}
We develop two procedures that provide such a guarantee.

Classical FDR literature focuses on rejecting a subset of hypotheses $R \in \{0, \, ..., \, m\}$ such that $R$ contains few false discoveries. \citet{benjamini1995controlling} showed that, in the context of \eqref{eq:setup}, we can control the FDR as follows. Let $p_{(1)}$, ..., $p_{(m)}$ be the sorted list of $p$-values, and let
$$ \hl_\alpha = \max\left\{l : p_{(l)} \leq \frac{\alpha \, l}{m}\right\}. $$
Then, if we reject those hypotheses corresponding to $\hl_\alpha$ smallest $p$-values, we control the FDR at level $\alpha$. This method for selecting the rejection set $R$ is known as the BH procedure.
The key difference between the setup of \citet{benjamini1995controlling} and our problem is that, in the former, the rejection set $R$ can be arbitrary, whereas here we must always reject the first $k$ hypotheses for some $k$. For example, even if the $p$-value corresponding to the third hypothesis is very small, we cannot reject the third hypothesis unless we also reject the first and second hypotheses.

\subsection{A BH-Type Procedure for Ordered Selection}

The main motivation behind our first procedure---\emph{ForwardStop}---is the following thought experiment. Suppose that we could transform our $p$-values $p_1$, ..., $p_m$ into statistics $q_1 < ... < q_m$, such that the $q_i$ behaved like a sorted list of $p$-values. Then, we could apply the BH procedure on the $q_i$, and get a rejection set $R$ of the form $R = \{1, \, ..., \, k\}$.

Under the global null where $p_1$, ..., $p_m \simiid U([0, 1])$, we can achieve such a transformation using the R\'enyi representation theorem \citep{renyi1953theory}. R\'enyi showed that if $Y_1$, ..., $Y_m$ are independent standard exponential random variables, then
$$ \left( \frac{Y_1}{m}, \, \frac{Y_1}{m} + \frac{Y_2}{m - 1}, \, ..., \, \sum_{i = 1}^m \frac{Y_i}{m - i + 1}\right)
\eqd E_{1, \, m}, \, E_{2, \, m}, \, ..., \, E_{m, \, m}, $$
where the $E_{i, \, m}$ are exponential order statistics, meaning that the $E_{i, \, m}$ have the same distribution as a sorted list of independent standard exponential random variables. R\'enyi representation provides us with a tool that lets us map a list of independent exponential random variables to a list of sorted order statistics, and vice-versa.

In our context, let
\begin{align}
\label{eq:Y}
&Y_i = -\log(1 - p_i), \\
\label{eq:Z}
&Z_i = \sum_{j = 1}^i Y_j \big/ (m - j + 1), \eqand \\
\label{eq:q}
&q_i = 1 - e^{-Z_i}.
\end{align}
Under the global null, the $Y_i$ are distributed as independent exponential random variables. Thus, by R\'enyi representation, the $Z_i$ are distributed as exponential order statistics, and so the $q_i$ are distributed like uniform order statistics.

This argument suggests that in an ordered selection setup, we should reject the first $\hkF^q$ hypotheses where
\begin{equation}
\label{eq:exact}
\hkF^q = \max\left\{k : q_k \leq \frac{\alpha \, k}{m}\right\}.
\end{equation}
The R\'enyi representation combined with the BH procedure immediately implies that the rule $\hkF$ controls the FDR at level $\alpha$ under the global null. Once we leave the global null, R\'enyi representation no longer applies; however, as we show in the following results, our procedure still controls the FDR.

We begin by stating a result under a slightly restricted setup, where we assume that the $s$ first $p$-values are non-null and the $m - s$ last $p$-values are null. We will later relax this constraint. The proof of the following result is closely inspired by the martingale argument of \citet{storey2004strong}.  As usual, our analysis is conditional on the non-null $p$-values (i.e., we treat them as fixed).

\begin{lemma}
\label{lemm:exact}
Suppose that we have $p$-values $p_1, \, ..., \, p_m \in (0, 1)$, the last $m - s$ of which are null, i.e., independently drawn from $U([0, 1])$. Define $q_i$ as in \eqref{eq:q}. Then
the  rule $\hkF^q$  controls the FDR at level $\alpha$, meaning that

\begin{equation}
\label{eq:ordered_FDR}
\EE{{\left(\hkF^q - s\right)_+}  \,\Big/\, { \max\left\{\hkF^q, 1\right\}  }} \leq \alpha. 
\end{equation}
\end{lemma}

Now the test statistics $q_i$ constructed in Lemma \ref{lemm:exact} depend on $m$. We can simplify the rule by augmenting our list of $p$-values with additional null test statistics (taking $m\rightarrow \infty$), and using the fact that $\frac{1 - e^{-x}}{x} \rightarrow 1$ as $x$ gets small. This gives rise to one of our main proposals:

\begin{proc}[ForwardStop]
\label{proc:forward}
Let $p_1, \, ..., \, p_m \in [0, 1]$, and let $0 < \alpha < 1$. We reject hypotheses $1, \, ..., \, \hkF$, where
\begin{align}
\label{eq:forward}
\hkF = &\max\left\{k \in \{1, \, ..., \, m\} : \frac{1}{k}\sum_{i = 1}^k Y_i \leq \alpha \right\},  \\
&\text{and } Y_i = -\log(1 - p_i). \notag
\end{align}
\end{proc}

We call this procedure \emph{ForwardStop} because it scans the $p$-values in a forward manner: If $\frac{1}{k}\sum_{i = 1}^k Y_i \leq \alpha$, then we know that we can reject the first $k$ hypotheses regardless of the remaining $p$-values. This property is desirable if we trust the first $p$-values more than the last $p$-values.

A major advantage of \emph{ForwardStop} over the direct R\'enyi stopping rule \eqref{eq:exact}
is that \emph{ForwardStop} provides FDR control even when some null hypotheses are interspersed
among the non-null ones. In particular, in the regression setting, this is important for achieving
FDR control for the incremental hypotheses $H_k^\inc$ \eqref{eq:incremental}, which are not
in general nested.

\begin{theo}
\label{theo:forward}
Suppose that we have $p$-values $p_1, \, ..., \, p_m \in (0, 1)$, a subset $N \subseteq \{1, \, ..., \, m\}$ are null, i.e., independently drawn from $U([0, 1])$. Then, the \emph{ForwardStop} procedure $\hkF$ \eqref{eq:forward} controls FDR at level $\alpha$, meaning that
$$ \EE{{\left|\left\{1, \, ..., \, \hkF \right\} \cap N\right|}  \,\Big/\, {\max\left\{\hkF, \, 1\right\}}} \leq \alpha. $$
\end{theo}

\subsection{Strong Control for Ordered Selection}
\label{sec:StrongStop}

In the previous section, we created the ordered test statistics $Z_i$ in \eqref{eq:Z} by summing transformed $p$-values starting from the first $p$-value. This choice was in some sense arbitrary. Under the global null, we could just as well obtain uniform order statistics $q_i$ by summing from the back:
\begin{align}
\label{eq:tY}
&\tY_i = -\log(p_i), \\
\label{eq:tZ}
&\tZ_i = \sum_{j = i}^m Y_j \big/ j, \eqand \\
\label{eq:tq}
&\tq_i = e^{-\tZ_i}.
\end{align}
If we run the BH procedure on these backward test statistics, we obtain another method for controlling
the number of false discoveries.

\begin{proc}[StrongStop]
\label{proc:strong}
Let $p_1, \, ..., \, p_m \in [0, 1]$, and let $0 < \alpha < 1$. We reject hypotheses $1, \, ..., \, \hk$, where
\begin{equation}
\label{eq:hkS}
\hkS = \max\left\{k\in\{1, \, \dots, \,m\}: \tq_k \leq \frac{\alpha k}{m}\right\}
\end{equation}
and $\tq_k$ is as defined in \eqref{eq:tq}.
\end{proc}

Unlike \emph{ForwardStop}, this new procedure needs to look at the $p$-values corresponding to the last hypotheses before it can choose to make any rejections.
This can be a liability if we do not trust the very last $p$-values much. Looking at the last $p$-values can however be useful if the model is correctly specified, as it enables us to strengthen our control guarantees: \emph{StrongStop} not only controls the FDR, but also controls the FWER.

\begin{theo}
\label{theo:strong}
Suppose that we have $p$-values $p_1, \, ..., \, p_m \in (0, 1)$, the last $m - s$ of which are null (i.e., independently drawn from $U([0, 1])$).  Then, the rule $\hkS$ from \eqref{eq:hkS} controls the FWER at level $\alpha$, meaning that
\begin{equation}
\label{eq:strong_guarantee}
\PP{\hkS > s} \leq \alpha. 
\end{equation}
\end{theo}

FWER control is stronger than FDR control, and so we immediately conclude from Theorem \ref{theo:strong} that \emph{StrongStop} also controls the FDR. Note that the guarantees from Theorem \ref{theo:strong} only hold when the non-null $p$-values all precede the null ones.

\section{Simulation Experiments: Simple Ordered Hypothesis Example}
\label{section:examples}

In this section, we demonstrate the performance of our methods in three
simulation settings of varying difficulty.  The simulation settings consist of
ordered hypotheses where the separation of the null and non-null hypotheses is
varied to determine the difficulty of the scenario.   Additional
simulations are provided in Appendix \ref{section:appendixsim}.

\begin{figure}[t]
  \centering
  \includegraphics[width=\textwidth]{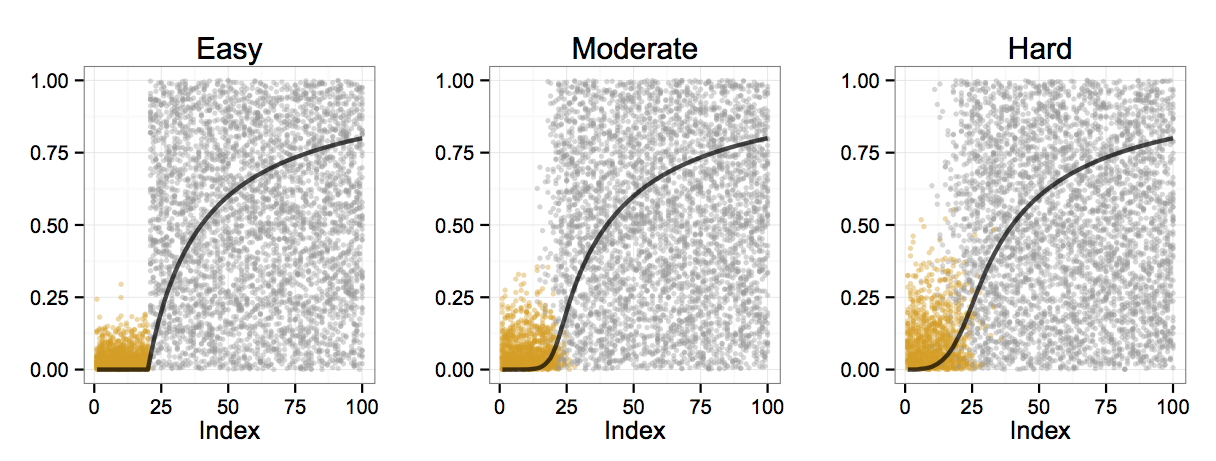}
  \caption{Observed $p$-values for $50$ realizations of the ordered hypothesis simulations
    described in Section \ref{section:examples}.  $p$-values corresponding to non-null hypotheses are shown in orange, while those corresponding to null hypotheses are shown in gray.  The smooth black curve is the average proportion of null hypotheses up to the given index, and is shown to help gauge the difficulty of the problem.  This curve can be thought of as the FDR of a fixed stopping rule which always stops at exactly the given index.    Non-null $p$-values are drawn from a $\mathrm{Beta}(1,b)$ distribution, with $b=23, 14, 8$ for the easy, medium and hard settings, respectively.  }
  \label{fig:changepoint}
\end{figure}

We consider a sequence of $m = 100$ hypotheses of which $s=20$ are non-null.  The $p$-values corresponding to the non-null hypotheses are drawn from a $\mathrm{Beta}(1,\beta)$ distribution, while those corresponding to true null hypotheses are $U([0,1])$.  
At each simulation iteration, the indices of the true null hypotheses are
selected by sampling without replacement from the set $\{1, 2, \ldots, m=100\}$
with probability of selection proportional to $i^\gamma$.  In this scheme, lower indices have smaller probabilities of being selected.  We present
results for three simulation cases, which we refer to as `easy' (perfect separation), `medium' ($\gamma = 8$), and `hard' ($\gamma = 4$).  In
the easy setup, we have strong signal $b=23$ and all the non-null hypotheses precede the null hypotheses, so we
have perfect separation.  In the medium difficulty setup, $b=14$ and the null and non-null hypotheses
are lightly inter-mixed.  In the hard difficulty setup, $b=8$ and the two are much more
inter-mixed.

\begin{figure}[t]
  \centering
  \includegraphics[width=0.3\textwidth]{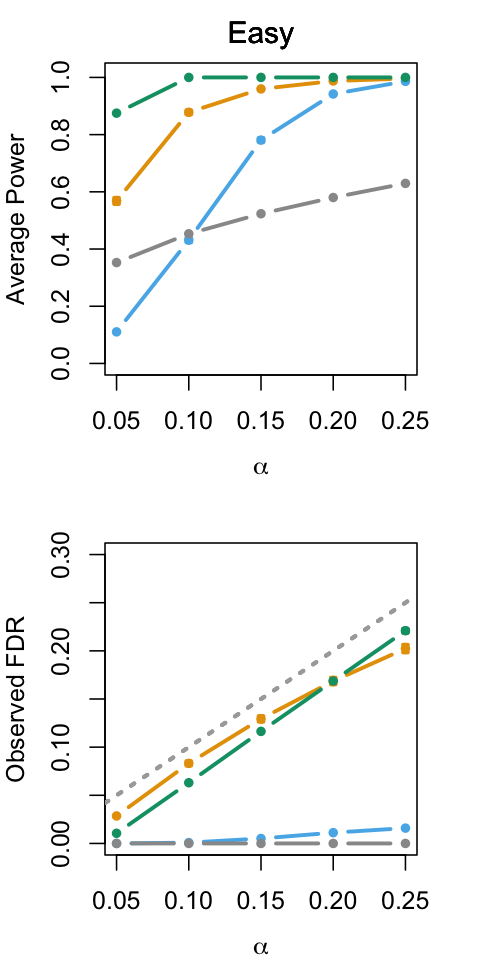} 
  \includegraphics[width=0.3\textwidth]{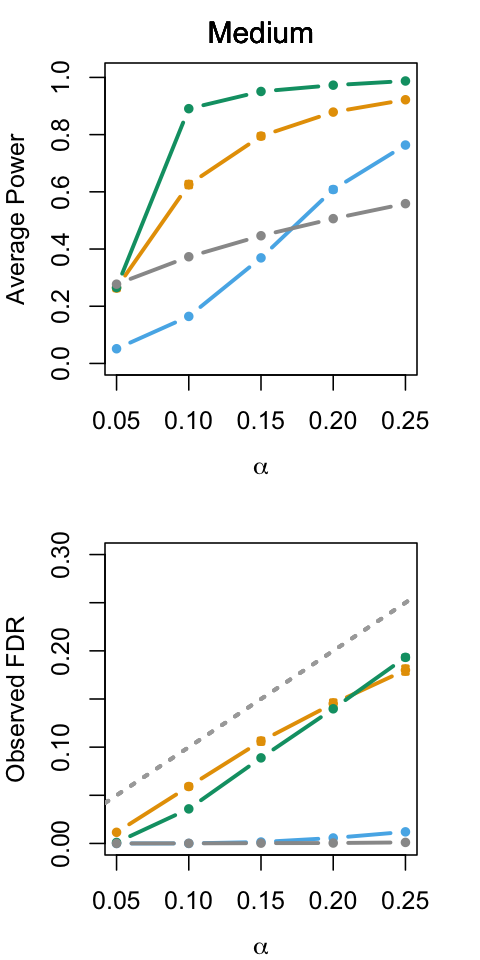}
  \includegraphics[width=0.3\textwidth]{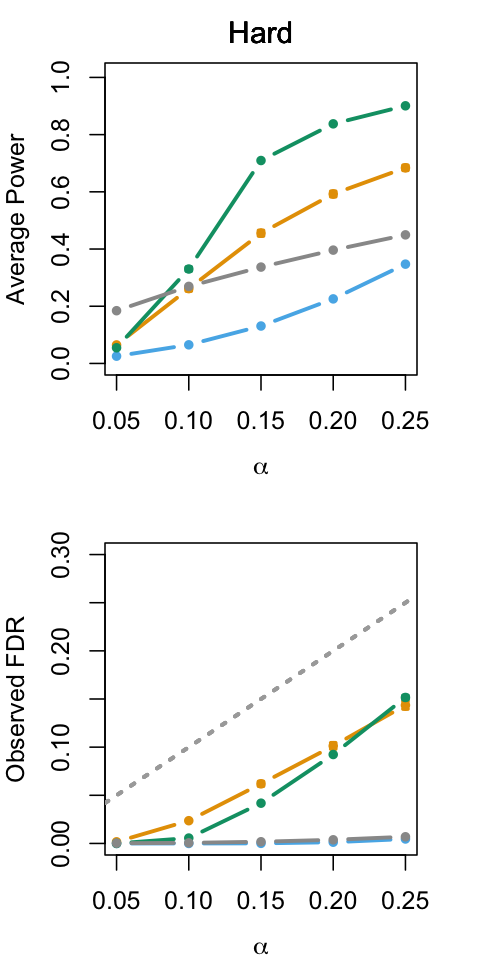} \\
  \includegraphics[width=0.75\textwidth]{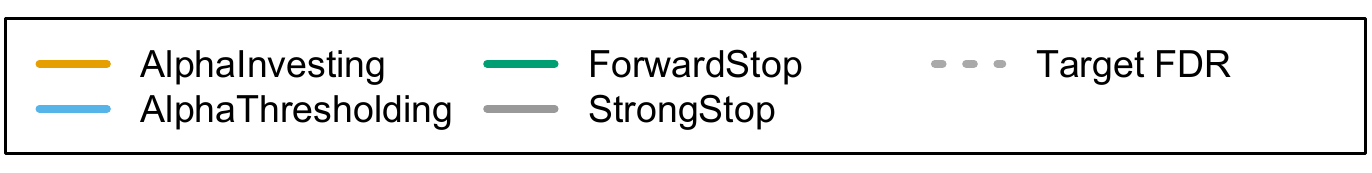}
  \caption{Average power and observed FDR level for the ordered hypothesis example based on 2000 simulation instances.  The notion of power used here is that of average power, defined as the fraction of non-null hypotheses that are rejected (i.e., $(k-V)/s$).   All four stopping rules successfully control FDR across the three difficulty settings.  \emph{StrongStop} and \emph{$\alpha$-thresholding} are both very conservative in terms of FDR control.  Even though \emph{ForwardStop} and \emph{$\alpha$-investing} have similar observed FDR curves, \emph{ForwardStop} emerges as the more powerful method, and thus has better performance in terms
of a precision-recall tradeoff. }
  \label{fig:changepoint_plots}
\end{figure}

For comparison, we also apply the following two rejection rules:
\begin{enumerate}
  \item \emph{Thresholding at $\alpha$.}  We reject all hypotheses up to the
    first time that a $p$-value exceeds $\alpha$.  This is guaranteed to
    control FWER and FDR at level $\alpha$ \citep{marcus1976closed}.
  \item \emph{$\alpha$-investing.}  We use the $\alpha$-investing scheme of
    \citet{foster2008alpha}.  While this procedure is not generally guaranteed
    to yield rejections that obey the ordering restriction, we can select
    parameters for which it does.  In particular, defining an investing rule
    such that the wealth is equal to zero at the first failure to reject, we get
    \begin{align*}
      \hat{k}_{invest} &= \min\left\{k: p_{k + 1} > \frac{(k + 1)\alpha}{1 + (k+1)\alpha}\right\}.
    \end{align*}
    This is guaranteed to control $\E V/(\E R + 1)$ at level $\alpha$.
    We note that, using generalized $\alpha$-investing \citep{aharoni2013generalized}, we could tweak the
$\alpha$-investing procedure to have more power to reject the earliest hypotheses and
less power for further ones; however, we will not explore that possibility here.
\end{enumerate}
These are the best competitors we are aware of for our problem.
We emphasize that, unlike
\emph{ForwardStop} and \emph{StrongStop}, these rules stop at the first $p$-value that
exceeds a given threshold. Thus, these methods will fail to identify true
rejections with very small $p$-values when they are preceded by a few
medium-sized $p$-values.

Figure \ref{fig:changepoint} shows scatterplots of observed $p$-values for 50
realizations of the three setups.  Figure \ref{fig:changepoint_plots} summarizes the performance of the four stopping rules.  We note that \emph{StrongStop} appears to be more powerful than other methods
weak signal/low $\alpha$ settings.  This may occur because, unlike the other methods,
\emph{StrongStop} scans $p$-values back-to-front and is therefore less sensitive
to the occurrence of large $p$-values early in the alternative. 

\section{Model Selection and Ordered Testing}
\label{section:application}

We now revisit the application that motivated our ordered hypothesis testing formalism.
As discussed in Section \ref{sec:path}, we assume that a path-based regression
procedure like forward stepwise regression or least-angle regression has given us
a sequence of models $\emptyset = \mm_0 \subset \mm_1 \subset \ldots \subset \mm_p$,
and our task is is to select one of these nested models. This results in an ordered hypothesis
testing problem that is \emph{conditional} on the order in which the regression algorithm adds
variables along its path.

We gave two options for formalizing the hypothesis that $\mm_k$ improves over $\mm_{k-1}$
and that the $k$-th variable should be added to the model: the incremental null \eqref{eq:incremental}
and the complete null \eqref{eq:complete}. In this section, we review recent proposals by
\citet{taylor2014post} and \citet{lockhart2013significance} for testing each of these nulls
in the case of least-angle regression and the lasso respectively,
and show how to incorporate them into our framework.

We emphasize again that the field
of ordered hypothesis testing appears to be growing rapidly, and that the applicability of our
sequential FDR controlling procedures is not limited to the tests surveyed here; for example,
if we wanted to test $H_k^\comp$ for forward stepwise regression or the graphical lasso,
we could use the test statistics of \citet{loftus2014significance} or \citet{gsell2013adaptive}
respectively.

\subsection{Testing the Incremental Null for Least-Angle Regression}
\label{section:spacingstest}

In the context of least-angle regression, \citet{taylor2014post} provide exact,
finite sample $p$-values for $H_k^\inc$ for generic design matrices $X$. The corresponding
test statistic is called the {\em spacing test}.
The first spacing test statistic $T_1$ has a simple form
\begin{equation}
\label{eq:lasso:test0}
T_1 = 
\left(1  - \Phi\left(\frac{\lambda_{1}}{\sigma }\right)\right) \, \Big / \,
\left(1 - \Phi\left(\frac{\lambda_{2}}{\sigma }\right)\right), 
\end{equation}
where $\lambda_1$ and $\lambda_2$ are the first two knots along the least-angle
regression path and $\sigma$ is the noise scale.
Given a standardized design matrix $X$ and the null hypothesis $H_1^k$, $T_1$ is uniformly distributed
over $[0, \, 1]$. Remarkably, this result holds under general position conditions on $X$ that hold
almost surely if $X$ is drawn from a continuous distribution, and
does not require $n$ or $p$ to be large.

\citet{taylor2014post} also derive similar test statistics $T_k$ for subsequent steps along
the least-angle regression path, which can be used for testing $H_K^\inc$.
Assuming Gaussian noise, all the $H_k^\inc$-null $p$-values
produced by this test are 1-dependent and uniformly distributed over $[0, \, 1]$.
For the purpose of our demonstrations, we apply our general FDR
control procedures directly as though the $p$-values were independent.
Developing a version of the spacing test that yields independent $p$-values remains an
active area of research.

\subsection{Testing the Complete Null for the Lasso}
\label{section:covtestbackground:lasso}

We also apply our formalism to testing the complete null for
the lasso path, using the covariance test statistics of \citet{lockhart2013significance}.
As our experiments will make clear, an advantage of testing the complete null
instead of the incremental null is the substantial increase in power.

In the case of orthogonal $X$, the covariance test statistics have the particularly simple form
\begin{equation}
  \label{eq:knotstatistic}
  T_k = \lambda_k(\lambda_k-\lambda_{k+1}),
\end{equation}
where $\lambda_1\ge\lambda_2\ge\dots$ denote the knots of the lasso path. Because
$X$ is orthogonal, the lasso never removes variables along its path, and so we know there
will be exactly $p$ knots.
\citet{lockhart2013significance} show that these test statistics satisfy the following asyptotic
guarantee. Recalling that the complete hypotheses are nested, suppose that
$H_1^\comp$, ..., $H_s^\comp$ are false, and $H_{s+1}^\comp$ is true.
Then, in the limit with $s$ fixed $n, \, p \rightarrow \infty$,
\begin{equation}
\label{eq:lasso_test}
\left(T_{s+1}, \, ... , \, T_{s+\ell}\right) \Rightarrow \left( \Exp(1), \,
\Exp\left(\frac{1}{2}\right), \, ..., \, \Exp\left(\frac{1}{\ell}\right)\right)
\end{equation}
for any fixed $\ell \geq 1$. As shown below, we can use the harmonic asymptotics
of these test statistics to improve the power of our sequential procedures.

The major limitation of the statistics \eqref{eq:knotstatistic} is that their distribution
can only be controlled asymptotically, and for orthogonal $X$. \citet{lockhart2013significance}
also provide adaptations of \eqref{eq:knotstatistic} that hold for non-orthogonal $X$;
however, the required asymptotic regime is then quite stringent so we may prefer to
use finite-sample-exact tests of \citet{taylor2014post} discussed in Section \ref{section:spacingstest}.
In the future, it may be possible to use ideas from \citet{fithian2014optimal} to 
devise non-asymptotic and powerful tests of $H_k^\comp$ for generic $X$.


\subsubsection{False Discovery Rate Control for Harmonic Test Statistics}
\label{section:tailstop}

Motivated by the harmonic form of the test statistics $T_k$ in \eqref{eq:lasso_test},
we show here how to improve the power of our sequential procedures in this setting.
Similar harmonic asymptotics also arise in other contexts, e.g., the test statistics for
the graphical lasso of \citet{gsell2013adaptive}.

Abstracting away from concrete regression problems, suppose that we have a sequence of arbitrary
statistics $T_1, \, ..., \, T_m \geq 0$ corresponding to $m$ hypotheses. The
first $s$ test statistics correspond to signal variables; the subsequent ones
are independently distributed as
\begin{equation}
\label{eq:lasso_ideal}
\left (T_{s+1}, \, ..., \, T_m\right) \sim \left( \Exp(1), \, \Exp\left(\frac{1}{2}\right), \, ..., \, \Exp\left(\frac{1}{m-s}\right)\right),
\end{equation}
where $\Exp(\mu)$ denotes the exponential distribution with mean $\mu$.
As before, we wish to construct a stopping rule that controls the FDR. 

To apply either {\em ForwardStop} or {\em StrongStop} 
using $p$-values based on \eqref{eq:lasso_ideal} would require knowledge of the
number of signal variables $s$, and hence would not be practical.
Fortuitously, however, an extension of this idea yields a variation of {\em StrongStop}
that does not require knowledge of $s$ and controls FDR.
Under \eqref{eq:lasso_ideal}, we have $j\cdot T_{s+j} \sim \Exp(1)$.
Using this fact, suppose that we knew $s$ and formed the {\em StrongStop}
rule for the $m-s$ null test statistics.
This would suggest a test based on
\begin{equation}
\label{eq:qStrongStopScaled}
q^*_i=\exp\left[-\sum_{j=i}^m \frac{\max\{1, \, j-s\}}{j} \, T_j\right ]
\end{equation}
This is not a usable test, since it depends on knowledge of $s$.
Now suppose we set $s=0$, giving
\begin{equation}
\label{eq:qTail}
q^*_i = \exp\left[-\sum_{j = i}^m T_j\right]
\end{equation}
 An application of the BH procedure to the $q_i^*$ leads to the following rule.

\begin{proc}[TailStop]
\label{proc:tailStop}
Let $q_i^*$ be defined as in \eqref{eq:qTail}. We reject hypotheses $1, \, ..., \, \hkT$, where
\begin{equation}
\label{eq:TailStop}
\hkT = \max\left\{k : q^*_k \leq \frac{\alpha k}{m}\right\}. 
\end{equation}
\end{proc}
Now the choice $s=0$ is anti-conservative (in fact, it is the least conservative possibility for $s$), and so as expected we lose
the strong control property of {\em StrongStop}.
But surprisingly, in the idealized setting of (\ref{eq:lasso_ideal}), {\em TailStop}  controls the FDR nearly exactly.

\begin{theo}
\label{theo:tailStop}
Given \eqref{eq:lasso_ideal}, the rule from \eqref{eq:TailStop} controls FDR at level $\alpha$. More precisely,
$$\EE{{\left(\hkT - s\right)_+}\,\Big/\,{ \max\left\{\hkT, 1\right\}  }} = \alpha \, \frac{m - s}{m}. $$
\end{theo}

The name \emph{TailStop} emphasizes the fact that this procedure starts
scanning the test statistics from the back of the list, rather than from the
front. Scanning from the back allows us to adapt to the harmonic decay of the
null $p$-values without knowing the number $s$ of non-null predictors. An
analogue to \emph{ForwardStop} for this setup would be much more difficult to
implement, as we would need to estimate $s$ explicitly. We emphasize that
the guarantees from Theorem \ref{theo:tailStop} hold under the generative model
\eqref{eq:lasso_ideal}, whereas the covariance test statistics only have this distribution
asymptotically. However, in our simulation experiments, the asymptotic regime appears
to hold well enough for this not to be an issue.

\section{Model Selection Experiments}
\label{section:appsimulation}

\begin{figure}[t]
  \centering
  \includegraphics[width=\textwidth]{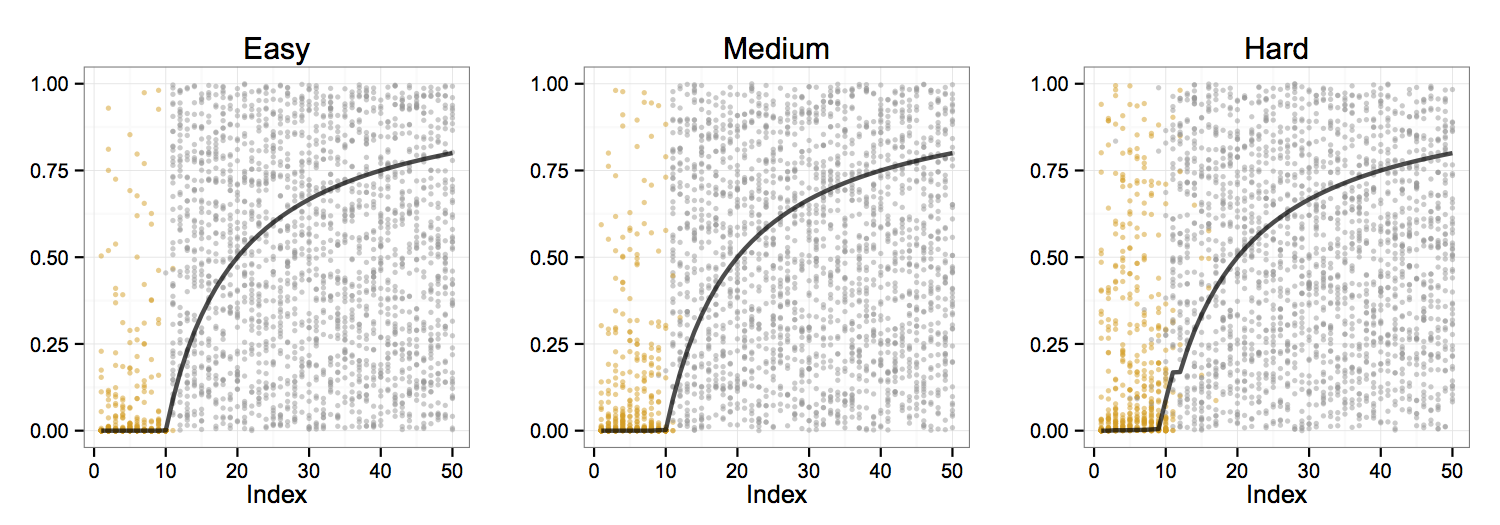}
  \caption{Observed $p$-values for $H_k^\inc$ in $50$ realizations of the spacing test \citep{taylor2014post} for
    least-angle regression.  $p$-values corresponding to non-null hypotheses are shown in orange, while those corresponding to null hypotheses are shown in gray.  The smooth black curve is the average proportion of null hypotheses up to the given index.  This example is similar to the Easy setting of the ordered hypothesis example of \textsection\ref{section:examples} in that the null and alternative are nearly perfectly separated.  However, in the least-angle regression setting the $p$-values under the alternative are highly variable and can be quite large, particularly in the Hard setting.   }
  \label{fig:lars_pvalues}
\end{figure}

In this section, we use the sequential procedures from \ref{section:seqtheory} for
pathwise model selection in sparse regression. As discussed in Section
\ref{section:application}, we focus on two particular problems: testing the incremental
null for least-angle regression with generic design (Section \ref{sec:gen_sim}), and
testing the complete null for the lasso with orthogonal design (Section \ref{sec:orthog_sim}).

The first of these two settings is of course more
immediately relevant to practice, and we verify that \emph{ForwardStop}
paired with the spacing test statistics of \citet{taylor2014post} performs well
on a real medical dataset. Meanwhile, the orthogonal simulations in Section
\ref{sec:orthog_sim} showcase the power boost that we can obtain from testing
the complete null instead of the incremental null. We believe that further
theoretical advances in the pathwise testing literature will enable us to have
similar power along with FDR guarantees in finite sample with generic $X$.

Finally, although our testing procedures are mathematically motivated by different
null hypotheses, namely the incremental and complete ones, we evaluate the performance
of each method in terms of its full-model false discovery rate (that is, the fraction of selected
variables that do not belong to the support of the true $\beta^*$). This lets us make a more
direct practical comparison between different methods.

\subsection{Testing the Incremental Null for Least-Angle Regression}
\label{sec:gen_sim}

\begin{figure}[t]
  \centering
  \includegraphics[width=0.3\textwidth]{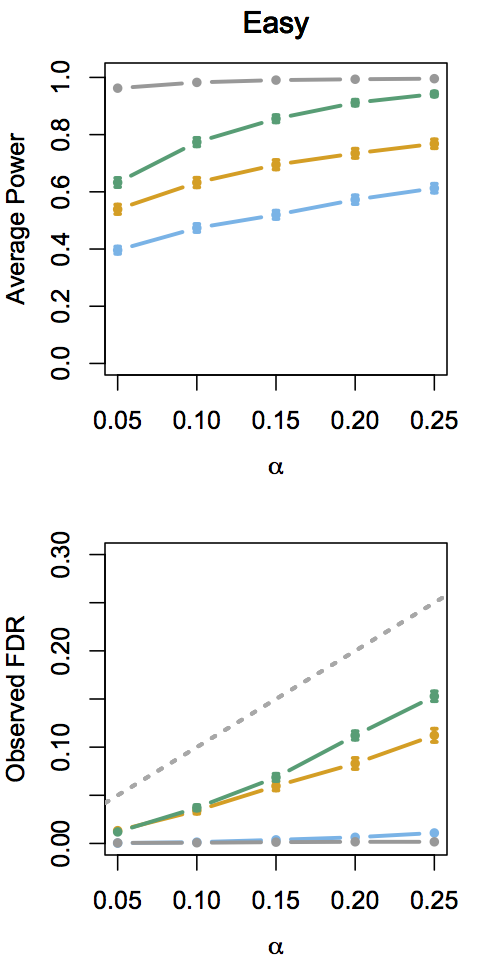} 
  \includegraphics[width=0.3\textwidth]{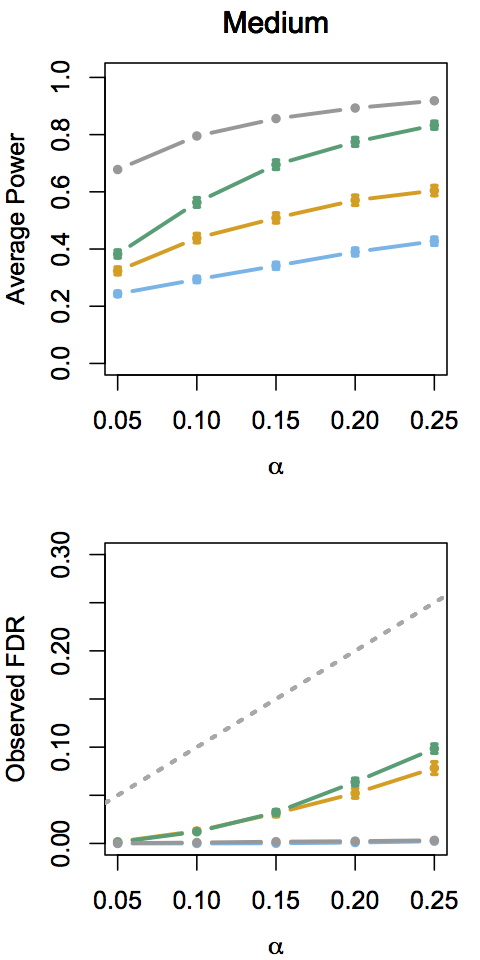}
  \includegraphics[width=0.3\textwidth]{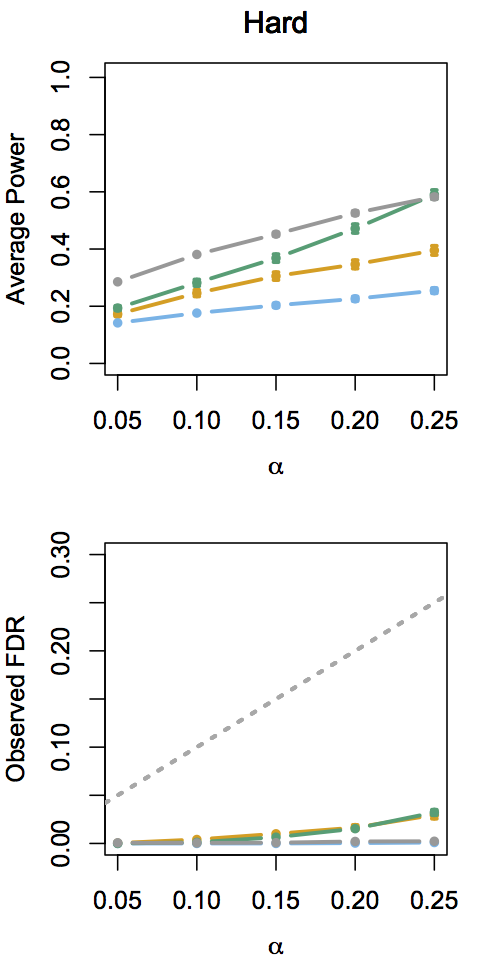} \\
  \includegraphics[width=0.75\textwidth]{./figures/main_legend.pdf}
  \caption{Average power and observed FDR level for the spacing test $p$-values for $H_k^\inc$ \citep{taylor2014post}. Even though there is nearly perfect separation between the null and alternative regions, the presence of large alternative $p$-values early in the path makes this a difficult problem.  \emph{StrongStop} attains both the highest average power and the lowest observed FDR across the simulation settings.  Unlike the other methods, \emph{StrongStop} scans $p$-values back-to-front, and is therefore able to perform well despite the occurrence of large $p$-values early in the path.}
  \label{fig:lars}
\end{figure}

 We compare the performance of \emph{ForwardStop}, \emph{StrongStop},
$\alpha$-investing and $\alpha$-thresholding on the spacing test statistics
from Section \ref{section:spacingstest}. We try three different simulation settings
with varying signal strength.  \emph{TailStop} is not included in this comparison
because it should only be used when the null test statistics exhibit harmonic
behaviour as in \eqref{eq:lasso_test}, whereas the spacing test $p$-values are uniform.  

In all three settings we have $n=200$ observations on $p=100$ variables of which $10$ are non-null, and standard normal errors on the observations.  The design matrix $X$ is taken to have iid Gaussian entries.  The non-zero entries of the parameter vector $\beta$ are taken to be equally spaced values from $2 \gamma$ to $\gamma\sqrt{2\log p}$, where $\gamma$ is varied to set the difficulty of the problem.  

Figure \ref{fig:lars_pvalues} shows $p$-values from 50 realizations of each simulation setting.  Note that while all three settings have excellent separation---meaning that least-angle regression selects most of the signal variables before admitting any noise variables---the $p$-values under the alternative can still be quite large.  Figure \ref{fig:lars} shows plots of average power and observed FDR level across the three simulation settings.  

\subsubsection{HIV Data}
\label{sec:hiv_data}

\begin{figure}[t]
  \centering
 \includegraphics[width=0.9\textwidth]{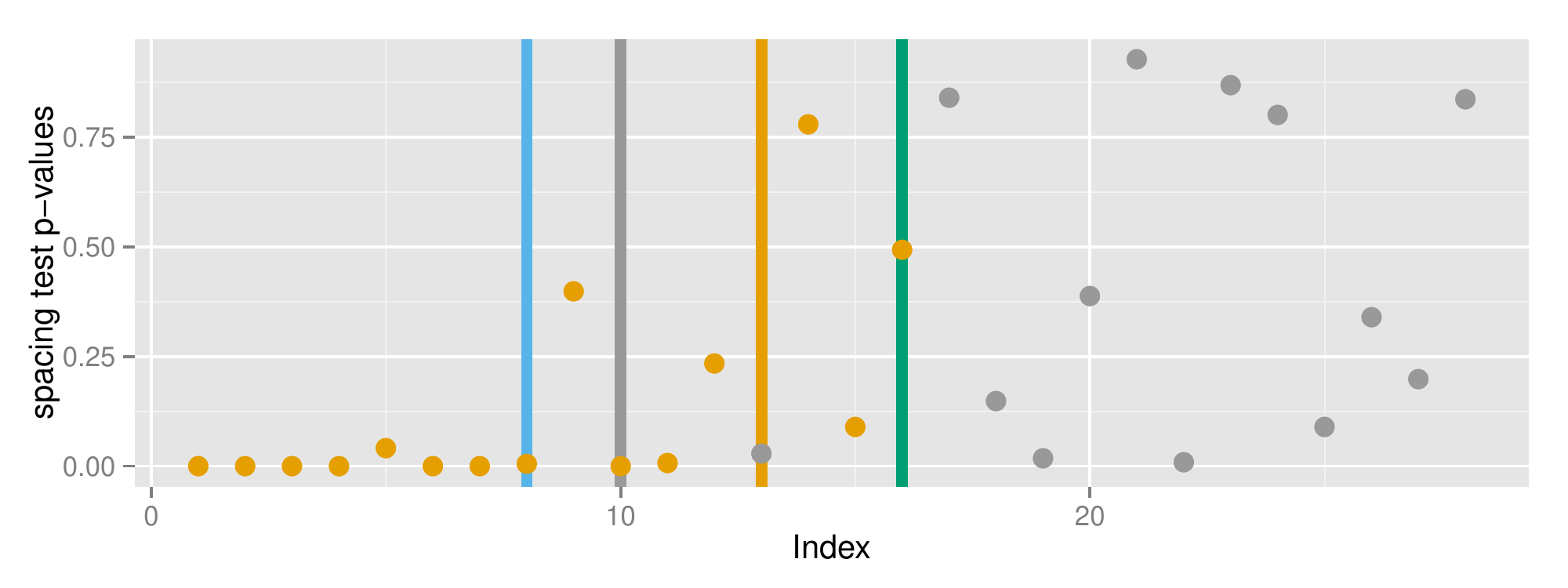}\\
 \includegraphics[width=0.75\textwidth]{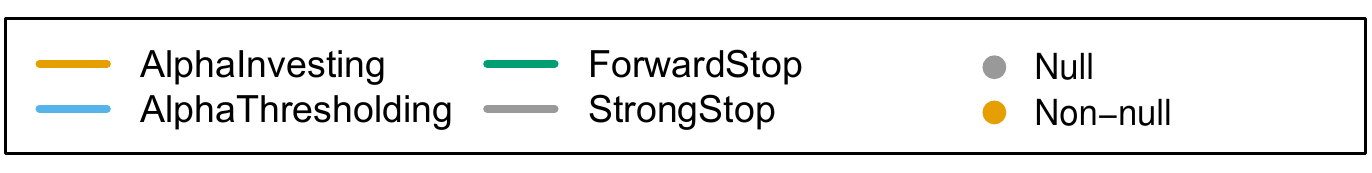}
  \caption{The $H_k^\inc$ $p$-values from the spacings test for the least-angle regression path, applied to the Abacavir resistance data of \citet{rhee2006genotypic}.  The vertical lines mark the stopping points of the four stopping rules, all with $\alpha=0.2$.  \emph{ForwardStop} selects the first 16 variables, even though the $p$-values at 9 and 14 are quite high.  All but one the selections made by \emph{ForwardStop} are considered meaningful by a previous study \citep{rhee2005hiv}.}
  \label{fig:dataexample}
\end{figure}

As a practical demonstration of our methods, we apply the same approach to the
Human Immunodeficiency Virus Type 1 (HIV-1) data of \citet{rhee2006genotypic},
which studied the genetic basis of HIV-1 resistance to several antiretroviral
drugs.  We focus on one of the accompanying data sets,
which measures the resistance of HIV-1 to six different Nucleoside RT
inhibitors (a type of antiretroviral drug) over 1005 subjects with mutations
measured at 202 different locations (after removing missing and duplicate
values).  The paper sought to determine which particular mutations were
predictive of resistance to these drugs.

In this section, we use least-angle regression to estimate a sparse linear model predicting drug
resistance from the mutation marker locations.  The \emph{ForwardStop}, \emph{StrongStop},
\emph{$\alpha$-investing}, and \emph{$\alpha$-thresholding} stopping rules are applied to the
$H_k^\inc$ $p$-values from the spacing test described in Section
\ref{section:spacingstest} to select a model along the least-angle regression path.  A previous
study of \citet{rhee2005hiv} provides a list of known relationships between
mutations and drug resistance, which allows partial assessment of the validity
of the selected variables.
This data set has also been studied by
\citet{barber2014controlling} in order to assess the performance of the knockoff filter
for variable selection; the main difference is that, unlike us, they do not constrain the
selection set to be the beginning of the least-angle regression path.

\newcommand{\wspp}{ {\;\;} }
\newcommand{\wsp}{\,}

\begin{table}
\caption{\label{tab:dataexample} Number of selections ($R$) made on the drug resistance data of \cite{rhee2006genotypic} using least-angle regression, the $p$-values from the spacings tests, and the four stopping rules (with $\alpha = 0.2$).  The number of the correctly selected mutation locations ($S$) is assessed using results from a previous study.  \emph{ForwardStop} and \emph{StrongStop} have the most competitive power, with the advantage varying by drug.  The abbreviations match those used in the original paper.}
\centering
  \begin{tabular}{|l||c|c||c|c||c|c||c|c||c|c||c|c|}
    \hline
 & \multicolumn{2}{c||}{\wspp\textbf{3TC}\wspp} & \multicolumn{2}{c||}{\wspp\textbf{ABC}\wspp} &
     \multicolumn{2}{c||}{\wspp\textbf{AZT}\wspp} & \multicolumn{2}{c||}{\wspp\textbf{D4T}\wspp} & 
         \multicolumn{2}{c||}{\wspp\textbf{DDI}\wspp} & \multicolumn{2}{c|}{\wspp\textbf{TDF}\wspp} \\
         \hline \hline
 \textbf{Rule} & \wsp $R$ \wsp & \wsp $S$ \wsp & \wsp $R$ \wsp & \wsp $S$ \wsp & \wsp $R$ \wsp & \wsp $S$ \wsp &
\wsp $R$ \wsp & \wsp $S$ \wsp & \wsp $R$ \wsp & \wsp $S$ \wsp & \wsp $R$ \wsp & \wsp $S$ \wsp  \\
 \hline \hline
    ForwardStop       & 4 & 4 & 16 & 15 & 4 & 4 & 18  & 14  & 3  & 3  & 6 & 6  \\
    StrongStop        & 4 & 4 & 10 & 10 & 8 & 8 & 10  & 9   & 12 & 12 & 6 & 6  \\
    $\alpha$-Thresholding & 4 & 4 & 8  & 8  & 4 & 4 & 10  & 9   & 3  &  3 & 2 & 2  \\
    $\alpha$-Investing    & 4 & 4 & 13 & 12 & 4 & 4 & 21  & 14  & 3  &  3 & 2 & 2 \\
    \hline
  \end{tabular}
\end{table}

Table \ref{tab:dataexample} shows the number of rejections and number of
correct rejections for each method applied to each of the six drug resistance
outcomes.  For illustration, we plot the $p$-values and stopping points for
resistance to Abacavir (ABC) in Figure \ref{fig:dataexample}.  The theory supporting
\emph{ForwardStop} and \emph{StrongStop} suggest that the selected models should contain
no more than 20\% false positives in expectation.  The
information available from the literature supports the validity of the
selections, showing that the variables selected by our procedures largely
corresponded to meaningful relationships based on previous studies conducted
with independent data.  We observe that, while all methods appear to achieve
overall FDR control, in each case either \emph{ForwardStop} or
\emph{StrongStop} yield the highest number of correct rejections.

\subsection{Testing the Complete Null for the Lasso with Orthogonal $X$}
\label{sec:orthog_sim}

In this section, we compare the performance of \emph{StrongStop},
\emph{ForwardStop}, $\alpha$-investing, $\alpha$-thresholding, as well
as \emph{TailStop} for testing the complete null for the lasso with orthogonal $X$
using the covariance test statistics of \citet{lockhart2013significance}.
As discussed in Section \ref{section:covtestbackground:lasso}, these test statistics
exhibit a harmonic behavior that \emph{TailStop} is designed to take advantage of.
 All other procedures operate
on conservative $p$-values, $p_j = \exp(-T_j)$, obtained by bounding the null
distributions by $\Exp(1)$.

\begin{figure}[t]
  \centering
  \includegraphics[width=\textwidth]{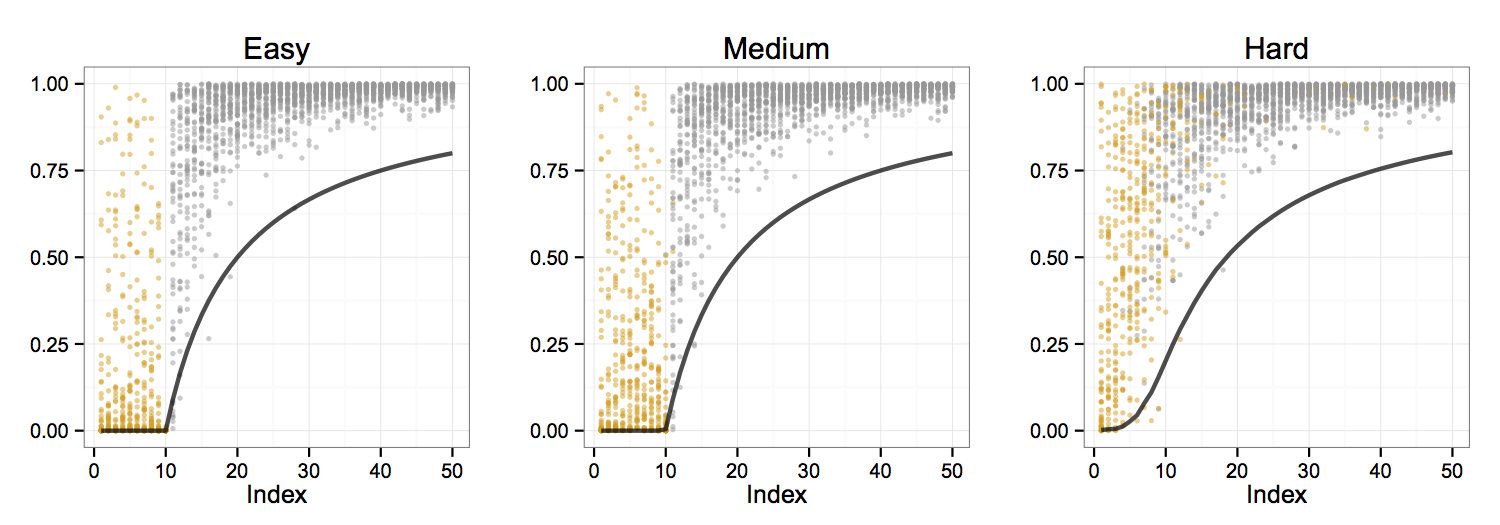}
  \caption{Observed $p$-values for $50$ realizations of the covariance test
    \citep{lockhart2013significance} for $H_k^\comp$ with orthogonal $X$.
    $p$-values corresponding to non-null hypotheses are shown in orange, while
    those corresponding to null hypotheses are shown in gray.  The smooth black
    curve is the average proportion of null hypotheses up to the given index.
    Note that these $p$-values behave very differently from those in the
    ordered hypothesis example presented in \textsection
    \ref{section:examples}.  The null $p$-values here exhibit $\Exp(1/\ell)$
    behaviour, as described in
    \textsection\ref{section:covtestbackground:lasso}.  Note that the non-null
    $p$-values for this test can be quite large on occasion.  TailStop performs
    well in part because it is not sensitive to the presence of some large
    non-null $p$-values.
  }
  \label{fig:ortho_lasso_pvalues}
\end{figure}

We consider three scenarios which we once again refer to as easy, medium and hard.  In all of the settings we have $n=200$ observations on $p=100$ variables of which $10$ are non-null, and standard normal errors on the observations.  The non-zero entries of the parameter vector $\beta$ are taken to be equally spaced values from $2 \gamma$ to $\gamma\sqrt{2\log p}$, where $\gamma$ is varied to set the difficulty of the problem.  Figure \ref{fig:ortho_lasso_pvalues} shows $p$-values from 50 realizations of each simulation setting; they exhibit harmonic behavior as described in Section \ref{section:covtestbackground:lasso}.

Figure \ref{fig:ortho} shows plots of average power and observed FDR level across the three simulation settings.  The superior performance of \emph{TailStop} is both desirable and expected, as it is the only rule that can take advantage of the rapid decay of the test statistics in the null.

\begin{figure}[t]
  \centering
  \includegraphics[width=0.3\textwidth]{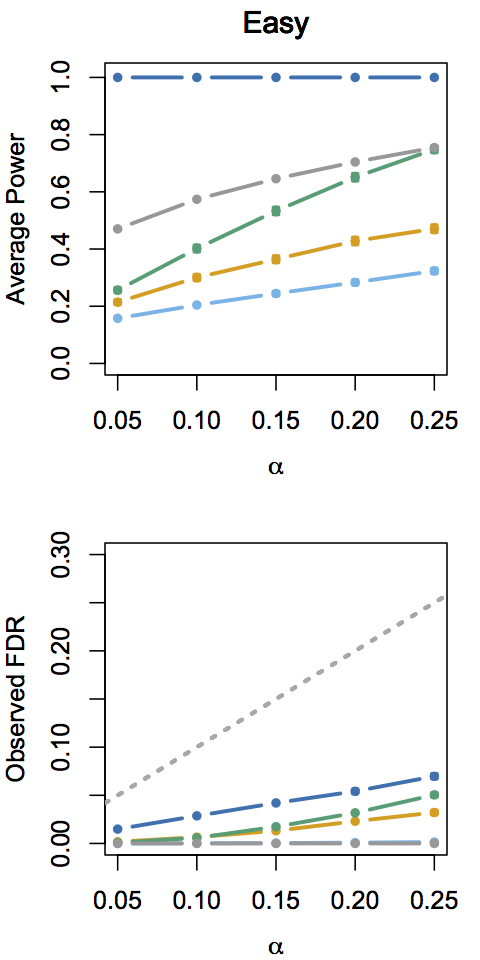} 
  \includegraphics[width=0.3\textwidth]{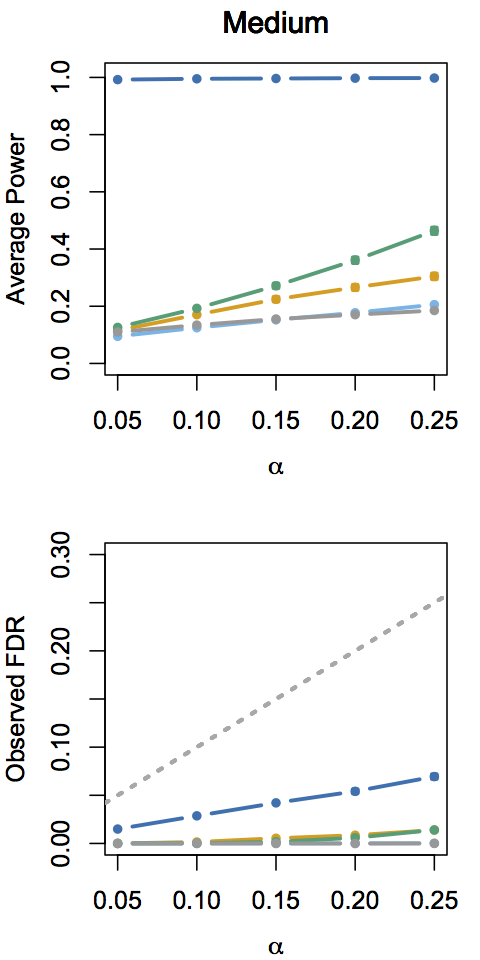}
  \includegraphics[width=0.3\textwidth]{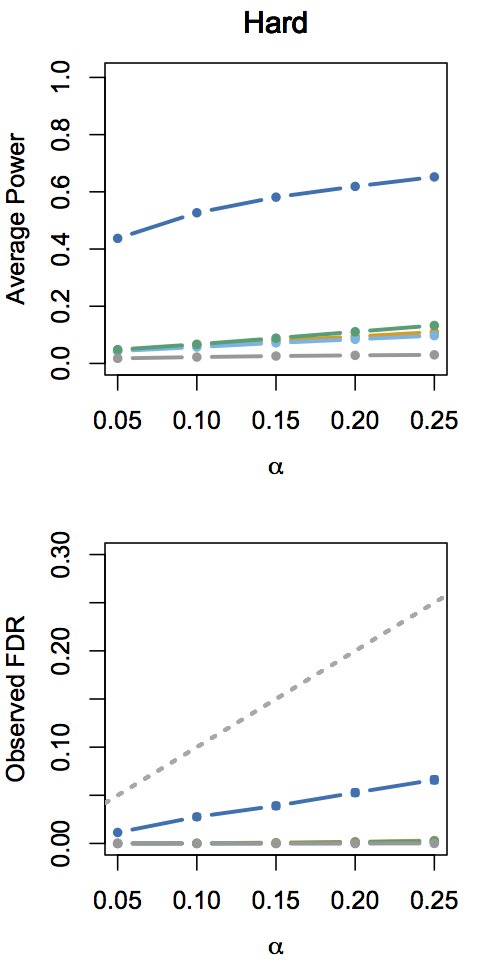} \\
  \includegraphics[width=0.75\textwidth]{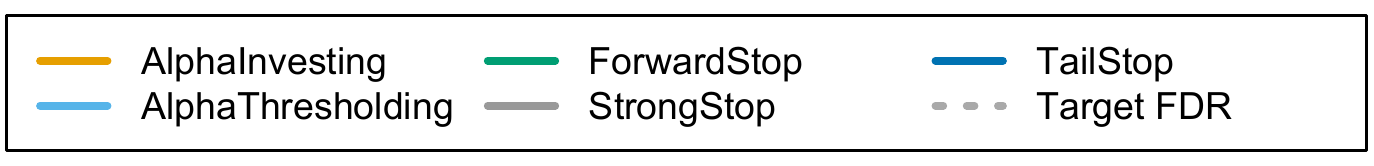}
  \caption{Average power and observed FDR level for the orthogonal lasso using
  the covariance test of \citet{lockhart2013significance} for $H_k^\comp$.  In the bottom
  panels, we see that all methods control the FDR.  However, in the medium and hard settings \emph{TailStop} is the only method that shows sensitivity to the choice of target $\alpha$ level.  All other methods have an observed FDR level that's effectively $0$, irrespective of the target $\alpha$.  From the power plots we also see that \emph{TailStop} has far higher power than
the other procedures --- in the medium setting at low $\alpha$ the power is almost 10 times higher than any other method.  By taking advantage of the
  $\Exp(1/\ell)$ behaviour of the null p-values, \emph{TailStop} far
  outperforms the other methods in power across all the difficulty settings. }
  \label{fig:ortho}
\end{figure}

\section{Conclusions}
\label{section:conc}

We have introduced a new setting for multiple hypothesis testing that is
motivated by sequential model selection problems.  In this setting, the
hypotheses are ordered, and all rejections are required to lie in an initial
contiguous block.  Because of this constraint, existing multiple testing
approaches do not control criteria like the False Discovery Rate (FDR).

We proposed a pair of procedures for testing in
this setting, denoted by \emph{ForwardStop} and \emph{StrongStop}.
We proved that these procedures control FDR at a specified
level while respecting the required ordering of the rejections.  Two procedures
were proposed because they provide different advantages. \emph{ForwardStop} is simple and robust to
assumptions on the particular behavior of the null distribution.
Meanwhile, when the null distribution is dependable, \emph{StrongStop}
controls not only FDR, but the Family-Wise Error Rate (FWER).  
We then applied our methods to model selection, and provided a modification
of \emph{StrongStop}, called \emph{TailStop}, which takes advantage of the
harmonic distributional guarantees that are available in some of those
settings.  

A variety of researchers are continuing to work on developing stepwise
distributional guarantees for a wide range of model selection problems.  As
many of these procedures are sequential in nature, we hope that the stopping
procedures from this paper will provide a way to convert these stepwise guarantees
into model selection rules with accompanying inferential guarantees.

There are many important challenges for future work.
For exact control of FDR or FWER, our methods require that the null $p$-values be independent. 
Except under orthogonal design, this is not true for any of the existing sequential $p$-value procedures that we are aware of.
Further work is need in extending our theory, and/or developing new sequential regression tests that yield independence under the null.

\section* {Acknowledgment}

M.G. and S.W. contributed equally to this paper.
The authors are grateful for helpful conversations with William Fithian and Jonathan Taylor,
and to the editors and referees for their constructive comments and suggestions.
M.G., S.W. and A.C. are respectively supported by
a NSF GRFP Fellowship,
a B.C. and E.J. Eaves SGF Fellowship, and
a NSERC PGSD Fellowship;
R.T. is supported by
NSF grant DMS-9971405 and NIH grant N01-HV-28183.
Most of this work was performed while M.G. and A.C. were
at the Stanford Statistics Department.

\bibliographystyle{Chicago}
\bibliography{references,tibs}

\appendix

\setcounter{theo}{4}

\section{Proofs}\label{section:appendixproof}

\begin{proof}[Lemma \ref{lemm:exact}]
We can map any rejection threshold $t$ to a number of rejections $k$. For the purpose of this proof, we will frame the problem as how to choose a rejection threshold $\hatt$; any choice of $\hatt \in [0, 1]$ immediately leads to a rule
$$\hkF = R(\hatt) = \left|\{i : q_i \leq \hatt\}\right|.$$
Similarly, the number of false discoveries is given by $V(\hatt) = \left|\{i > s : q_i \leq \hatt\}\right|$. We define the threshold selection rule
$$ \hatt_\alpha = \max\left\{t \in [0, 1] : t \leq \frac{\alpha \, R(t)}{m} \right\}.$$
Here, $R(\hatt_\alpha) = \hkF$ and so this rule is equivalent to the one defined in the hypothesis.

When coming in from $0$, $R(t)$ is piecewise continuous with upwards jumps, so
$$ \hatt_\alpha = \frac{\alpha \, R(\hatt_\alpha)}{m}, $$
allowing us to simplify our expression of interest:
$$\frac{V(\hatt_\alpha)}{R(\hatt_\alpha)} = \frac{\alpha}{m} \, \frac{ V(\hatt_\alpha)}{\hatt_\alpha}. $$
Thus, in order to prove our result, it suffices to show that 
$$ \EE{\frac{ V(\hatt_\alpha)}{\hatt_\alpha}} \leq m. $$
The remainder of this proof establishes the above inequality using R\'enyi representation and a martingale argument due to \citet{storey2004strong}.

Recall that, by assumption, $p_{s+1}, \, ..., \, p_m \simiid U([0, 1])$. Thus, we can use R\'enyi representation to show that
\begin{align*}
\left(Z_{s+1} - Z_s, \, ..., \, Z_m - Z_s\right)
&= \left( \frac{Y_{s+1}}{m - s}, \, ..., \, \sum_{i = s+1}^m \frac{Y_i}{m - i + 1} \right)\\
&\eqd \left(E_{1, \, m - s}, \, ..., \, E_{m - s, \, m - s} \right),
\end{align*}
where the $E_{i, \, m - s}$ are standard exponential order statistics, and so
$$ \left(e^{-(Z_{s+1} - Z_s)}, \, ..., \, e^{-(Z_m - Z_s)}\right) $$
are distributed as $m - s$ order statistics drawn from the uniform $U([0, 1])$ distribution. Recalling that
$$ 1 - q_{s + i} =  (1 - q_s) \, e^{-(Z_{s + i} - Z_s)}, $$
we see that $q_{s+1}$, ..., $q_m$ are distributed as uniform order statistics on $[q_s, \, 1]$.

Because the last $q_i$ are uniformly distributed,
$$ {M}(t) = \frac{V(t)}{t} $$
is a martingale on $(q_s, \, 1]$
with time running backwards. Here, the relevant filtration
$\ff_t$ tells us which of the $q_i$ are strictly greater than $t$; we can
also verify that $\hatt_{\alpha}$ is a stopping time with respect to
this backwards-time filtration.
Now, let ${M}^+(t)$,  $\hatt^+_{\alpha}$, and $\ff_t^+$ be the
right-continuous modifications of the previous quantities (again, with respect
to backwards-running time). By the optional sampling theorem
$$ \EE{\min \{M^+(\hatt^+_\alpha), \, C\} ; \hatt^+_\alpha > q_s} \leq M(1) = \frac{m - s}{1 - q_s} $$
for any $C \geq 0$; thus, by the (Lebesgue) monotone convergence theorem,
$$ \EE{M^+(\hatt^+_\alpha) ; \hatt^+_\alpha > q_s} \leq \frac{m - s}{1 - q_s} $$
Moreover, we can verify that 
$$  \EE{M^+(\hatt^+_\alpha) ; \hatt^+_\alpha > q_s} =  \EE{M(\hatt_\alpha) ; \hatt_\alpha > q_s}, $$
almost surely, and so
$$ \EE{M(\hatt_\alpha) ; \hatt_\alpha > q_s} \leq  \frac{m - s}{1 - q_s}. $$
For all $t > q_s$,
$$ \frac{V(t)}{t} = \frac{t - q_s}{t} M(t) \leq (1 - q_s) \, M(t), $$
and so 
$$ \EE{\frac{V(\hatt_\alpha)}{\hatt_\alpha} ; \hatt_\alpha > q_s} \leq m - s. $$
Meanwhile, 
$$ \EE{\frac{V(\hatt_\alpha)}{\hatt_\alpha} \big| \hatt_\alpha \leq q_s} = 0,
\text{ and so, as claimed, }
 \EE{\frac{V(\hatt_\alpha)}{\hatt_\alpha}} \leq m. $$
 \hfill \qed
\end{proof}

We begin our analysis of \emph{ForwardStop} (Procedure \ref{proc:forward}) by showing that
it satisfies the same guarantees as the stopping rule \eqref{eq:exact}. Although the following
corollary is subsumed by Theorem \ref{theo:forward}, its simple proof can still be helpful for understanding
the motivation behind \emph{ForwardStop}.

\begin{coro}
\label{coro:limit}
Under the conditions of Lemma \ref{lemm:exact}, 
the {\em ForwardStop} procedure defined in (\ref{eq:forward})
has FDR is controlled at level $\alpha$.
\end{coro}
\begin{proof}
We can extend our original list of $p$-values $p_1$, ..., $p_m$ by appending additional terms
$$ \tp_{m+1}, \, \tp_{m+2}, \,  ..., \, \tp_{m^*} \simiid U([0,1]) $$
to it. This extended list of $p$-values still satisfies the conditions of Lemma \ref{lemm:exact}, and so we can apply procedure \eqref{eq:exact} to this extended list without losing the FDR control guarantee:
$$ \hkF^{q, m^*} = \max\left\{k : \frac{m^* q_k^{m^*}}{k} \leq \alpha \right\}. $$
As we take $m^* \rightarrow \infty$, we have
$$ \lim_{m^* \rightarrow \infty} m^* q_k^{m^*} = \lim_{m^* \rightarrow \infty} m^* \left(1 - \exp\left[-\sum_{j = 1}^k \frac{Y_j}{m^* - j + 1}\right]\right) = \sum_{j = 1}^k Y_j, $$
and so, because the set $[0, \, \alpha]$ is closed, we recover the procedure described in the hypothesis:
$$ \lim_{m^* \rightarrow \infty} \hkF^{q,m^*} = \hkF. $$
Thus, by dominated convergence, the rule $\hkF$ controls the FDR at level $\alpha$. \hfill\qed
\end{proof}

\begin{proof}[Theorem \ref{theo:forward}]

The proof of Lemma \ref{lemm:exact} used quantities
$$ Z_i = \sum_{j = 1}^i \frac{Y_j}{m - j + 1} = \sum_{j = 1}^i \frac{Y_j}{|\{l \in \{j, \, ..., \, m\}\}|} $$
to construct the sorted test statistics $q_i$. The key difference between the setup of Lemma \ref{lemm:exact} and our current setup is that we can no longer assume that if the $i^{th}$ hypothesis is null, then all subsequent hypotheses will also be null.

In order to adapt our proof to this new possibility, we need to replace the $Z_i$ with
$$ Z_i^{ALT} = \sum_{j = 1}^i \frac{Y_j}{\nu(j)}, \quad \text{where } \nu(j) = |\{l \in \{j, \, ..., \, m\} : l \in N\}|, $$
and $N$ is the set of indices corresponding to null hypotheses. Defining
$$ q_i^{ALT} = 1 - e^{-Z_i^{ALT}}, $$
we can use R\'enyi representation to check that these test statistics have distribution
\begin{align*}
&1 - q_i^{ALT} \eqd r(i) \, \left(1 -U_{\nu(i), \, |N|}\right), \text{ where} \\
&r(i) := \exp\left[-\sum_{\{j \leq i : j \notin N\}} \frac{Y_j}{i} \right]
\end{align*}
and the $U_{\nu(j), \, |N|}$ are order statistics of the uniform $U([0, 1])$ distribution. Here $r(i)$ is deterministic in the sense that it only depends on the location and position of the non-null $p$-values.

If we base our rejection threshold $\hatt^{ALT}_\alpha$ on the $q_i^{ALT}$, then by an argument analogous to that in the proof of Lemma \ref{lemm:exact}, we see that
$$ \frac{V\left(\hatt^{ALT}_\alpha\right)}{\hatt^{ALT}_\alpha} $$
is a sub-martingale with time running backwards. The key step in showing this is to notice is that, now, the decay rate of $V(t)$ is accelerated by a factor $r^{-1}(i) \geq 1$. Thus, the rejection threshold $\hatt^{ALT}_\alpha$ controls FDR at level $\alpha$ in our new setup where null and non-null hypotheses are allowed to mix.

Now, of course, we cannot compute the rule $\hatt^{ALT}_\alpha$ because the $Z_i^{ALT}$ depend on the unknown number $\nu(j)$ of null hypotheses remaining. However, we can apply the same trick as in the proof of Corollary \ref{coro:limit}, and append to our list an arbitrarily large number of $p$-values that are known to be null. In the limit where we append infinitely many null $p$-values to our list, we recover the \emph{ForwardStop} rejection threshold. Thus, by dominated convergence, \emph{ForwardStop} controls the FDR even when null and non-null hypotheses are interspersed. \hfill \qed
\end{proof}

\begin{proof}[Theorem \ref{theo:strong}]

We begin by considering the global null case. In this case, the $\tY_i$ are all standard exponential, and so by R\'enyi representation the $\tq_i$ are distributed as the order statistics of a uniform $U([0, 1])$ random variable. Thus, under the global null, the rule $\hkS$ is just Simes' procedure \citep{simes1986improved} on the $\tq_i$. Simes' procedure is known to provide exact $\alpha$-level control under the global null, so \eqref{eq:strong_guarantee} holds as an equality under the global null.

Now, consider the case where the global null does not hold. Suppose that we have $\hkS = k > s$. From the definition of $\tq_k$, we see that $\tq_k$ depends only on $p_k, \, ..., \, p_m$, and so the event $\tq_k \leq \alpha k/m$ is just as likely under the global null as under any alternative with less than $k$ non-null $p$-values. Thus, conditional on $s$,
$$ \sum_{k = s+ 1}^m \PP{\hkS = k \big| \text{alternative}}  = \sum_{k = s+ 1}^m \PP{\hkS = k \big| \text{null}} \leq \alpha, $$
and so the discussed procedure in fact provides strong control. \hfill \qed
\end{proof}

\begin{proof}[Theorem \ref{theo:tailStop}]
Let $Z^*_i = \sum_{j = i}^m T_i$. By R\'enyi representation,
$$ \left( Z^*_{s+1}, \, ..., \, Z^*_m\right) \sim \left(E_{m - s, \, m - s}, \, ..., \, E_{1, \, m - s} \right), $$
where the $E_{i, \, j}$ are exponential order statistics. Thus, the null test statistics
$$ \left(q^*_{s + 1}, \, ..., \, q^*_m \right) $$
are distributed as $m - s$ order statistics drawn from the uniform $U([0, 1])$ distribution. The result of \citet{benjamini1995controlling} immediately implies that we can achieve FDR control by applying the BH procedure to the $q_i^*$, and so \emph{TailStop} controls the FDR. The exact equality follows from the result of \citet{BY01}. \hfill \qed
\end{proof}

\section{Additional Simulations}\label{section:appendixsim}

In this section we revisit the ordered hypothesis example introduced in Section \ref{section:examples} and present the results of a more extensive simulation study.  We explore the following perturbations of the problem:
\begin{enumerate}[(a)] \setlength{\itemsep}{-0.2em}
  \item Varying signal strength while holding the level of separation fixed. (Figures \ref{fig:cp_appendix_perfect}, \ref{fig:cp_appendix_good}, \ref{fig:cp_appendix_moderate})
  \item Increasing the number of hypotheses while retaining the same proportion of non-null hypotheses (Figure \ref{fig:cp_appendix_diffn})
  \item Varying the proportion of non-null hypotheses (Figures \ref{fig:cp_appendix_perfect_pi0}, \ref{fig:cp_appendix_good_pi0}, 
  \ref{fig:cp_appendix_moderate_pi0})
\end{enumerate}
We remind the reader of the three simulation settings introduced in \ref{section:examples}, which we termed Easy, Medium and Hard.  These settings were defined as follows
\begin{description} \setlength{\itemsep}{-0.2em}
    \item[Easy] Perfect separation (all alternative precede all null), and strong signal ($\mathrm{Beta}(1,23)$)
    \item[Medium]  Good separation (mild intermixing of hypotheses), and moderate signal ($\mathrm{Beta}(1,14)$)
    \item[Hard] Moderate separation (moderate intermixing hypotheses), and low signal ($\mathrm{Beta}(1,8)$)
\end{description}
All results are based on $2000$ simulation iterations.  Unless otherwise specified, the simulations are carried out with $m=100$ total hypotheses of which $s=20$ are non-null.  

\begin{figure}[p]
  \centering
  \includegraphics[width=0.75\textwidth]{./figures/main_legend.pdf} \\
  \includegraphics[width=0.29\textwidth]{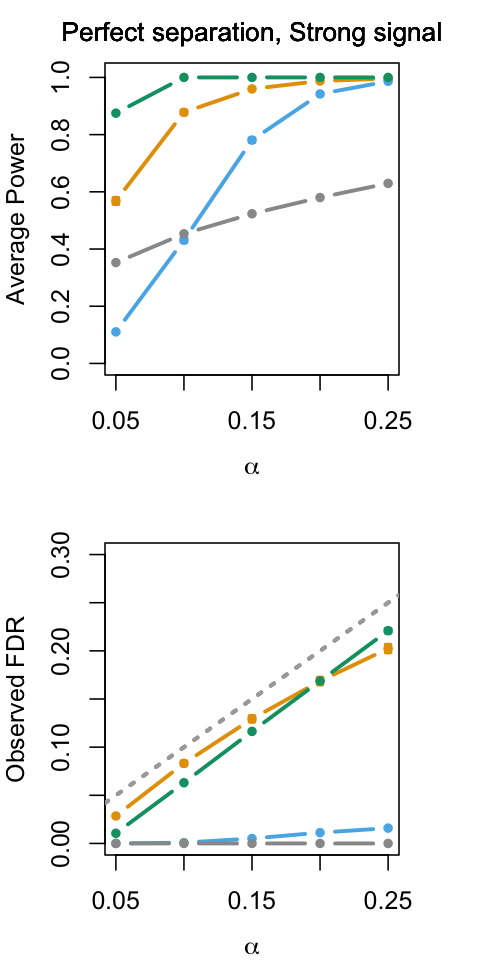} 
  \includegraphics[width=0.29\textwidth]{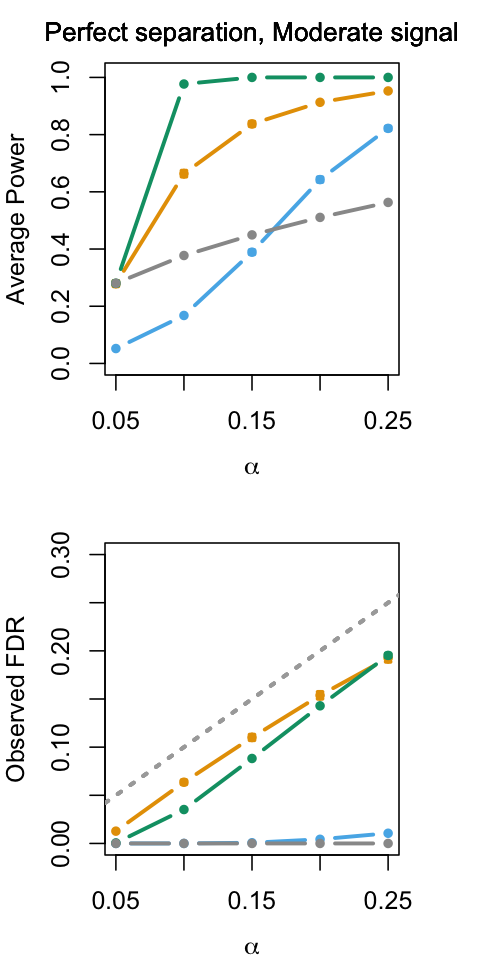}
  \includegraphics[width=0.29\textwidth]{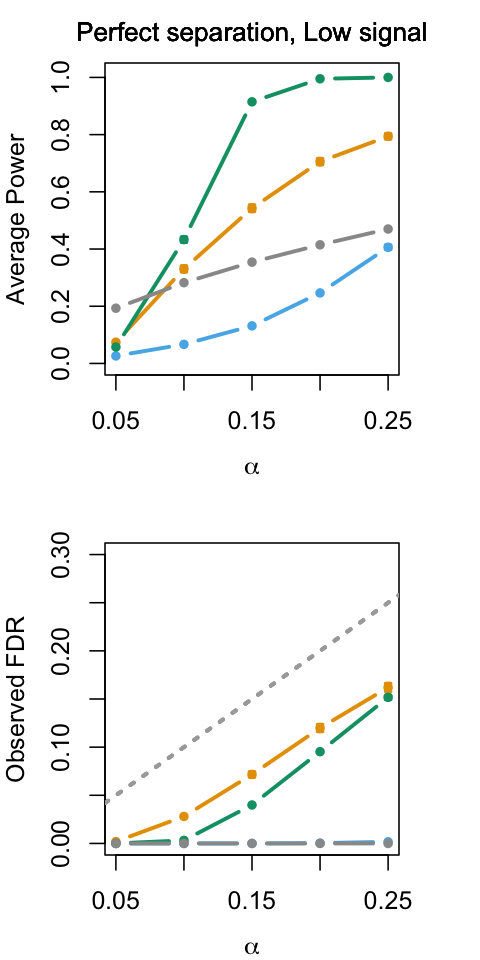}
  \caption{Effect of signal strength on stopping rule performance: Perfect separation regime. \emph{ForwardStop} remains the best performing method overall, except at the lowest $\alpha$ level in the moderate and low signal regimes.  All of the methods become more conservative as the signal strength decreases.}
  \label{fig:cp_appendix_perfect}
\end{figure}

\begin{figure}[p]
  \centering
  \includegraphics[width=0.3\textwidth]{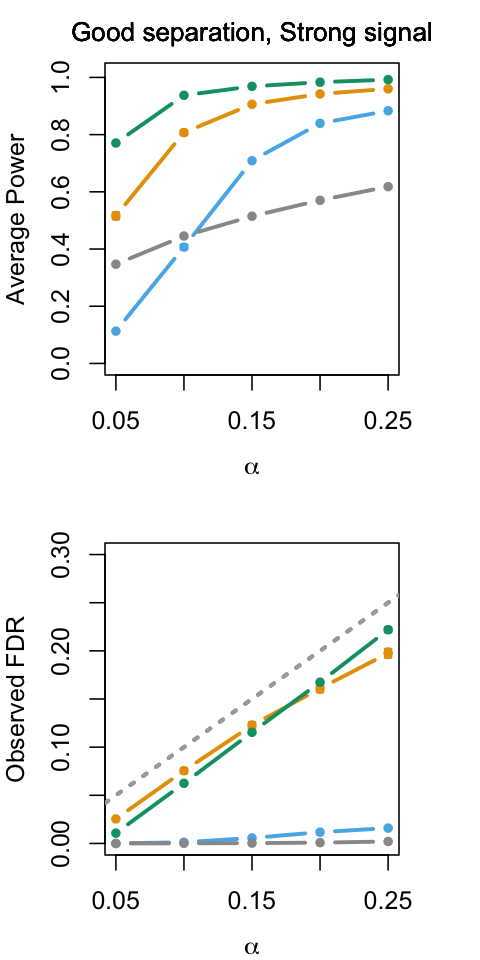} 
  \includegraphics[width=0.3\textwidth]{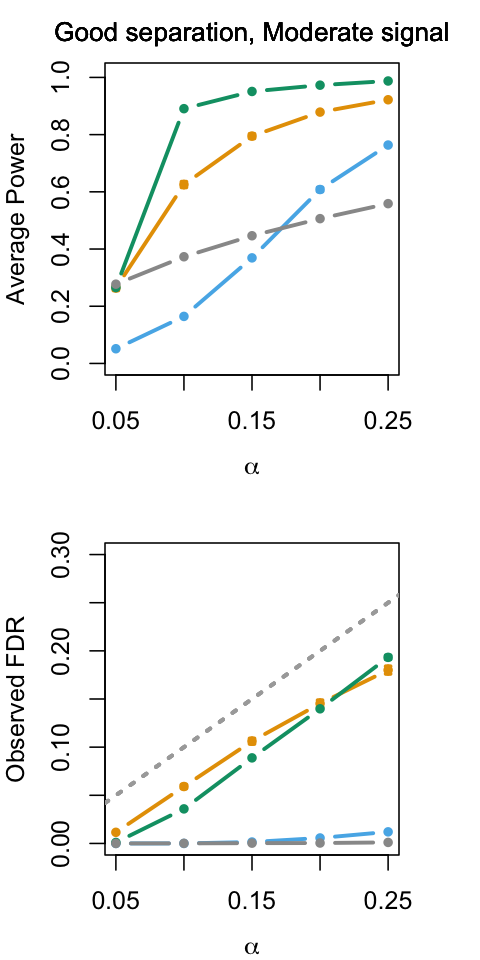}
  \includegraphics[width=0.3\textwidth]{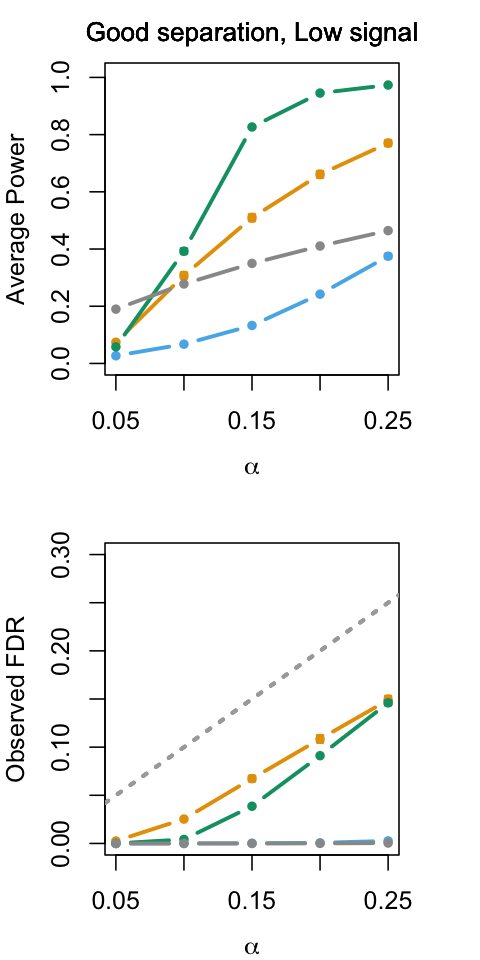}
  \caption{Effect of signal strength on stopping rule performance: Good separation regime. The effect of signal strength is qualitatively the same as in the perfect separation regime.  }
  \label{fig:cp_appendix_good}
\end{figure}

\begin{figure}[p]
  \centering
  \includegraphics[width=0.3\textwidth]{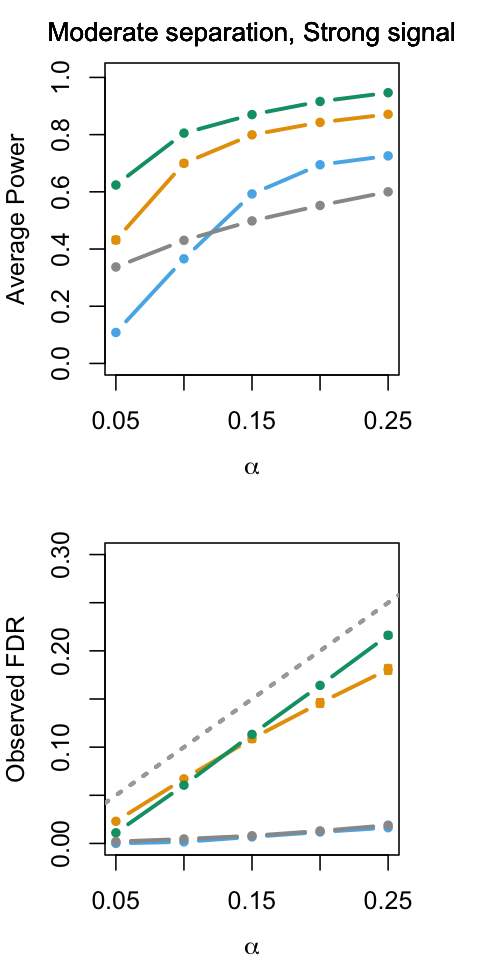} 
  \includegraphics[width=0.3\textwidth]{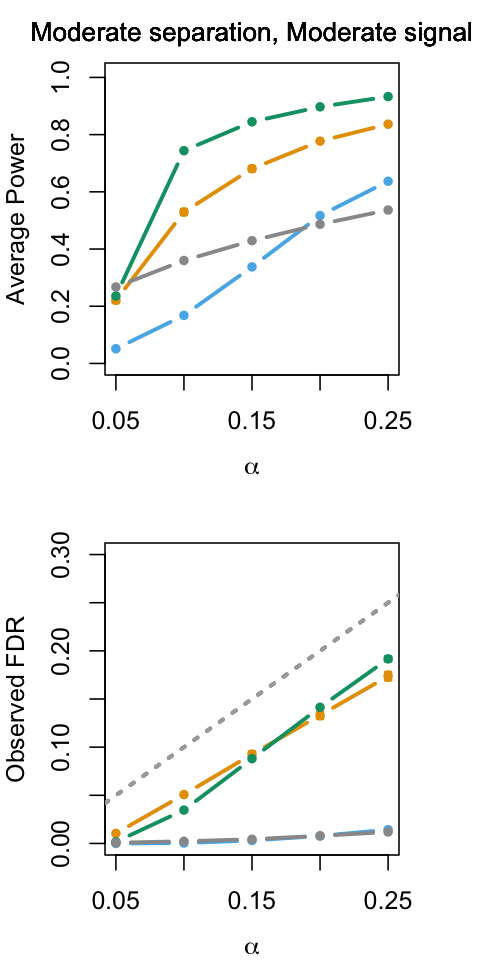}
  \includegraphics[width=0.3\textwidth]{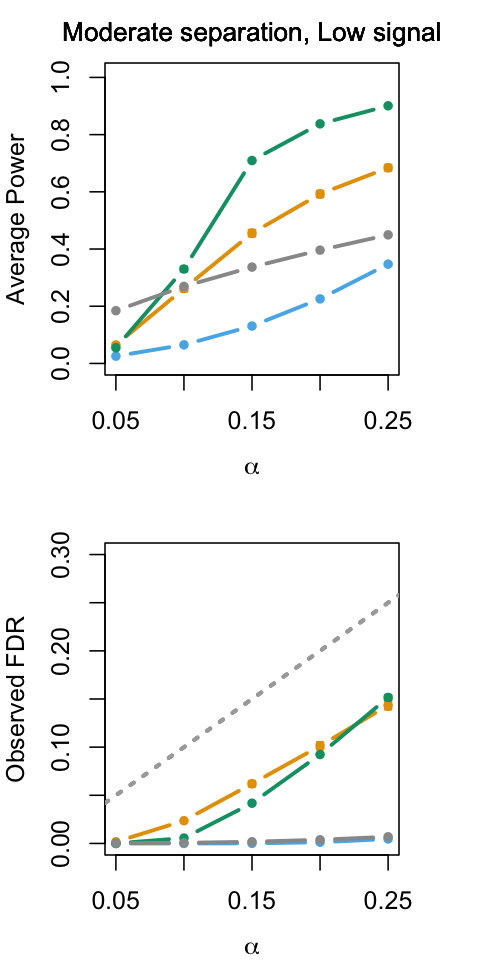}
  \caption{Effect of signal strength on stopping rule performance: Moderate separation regime. The effect of signal strength is qualitatively the same as in the perfect separation and good separation regimes. }
  \label{fig:cp_appendix_moderate}
\end{figure}

\begin{figure}[p]
  \centering
  \includegraphics[width=0.75\textwidth]{./figures/main_legend.pdf} \\
  \includegraphics[width=0.3\textwidth]{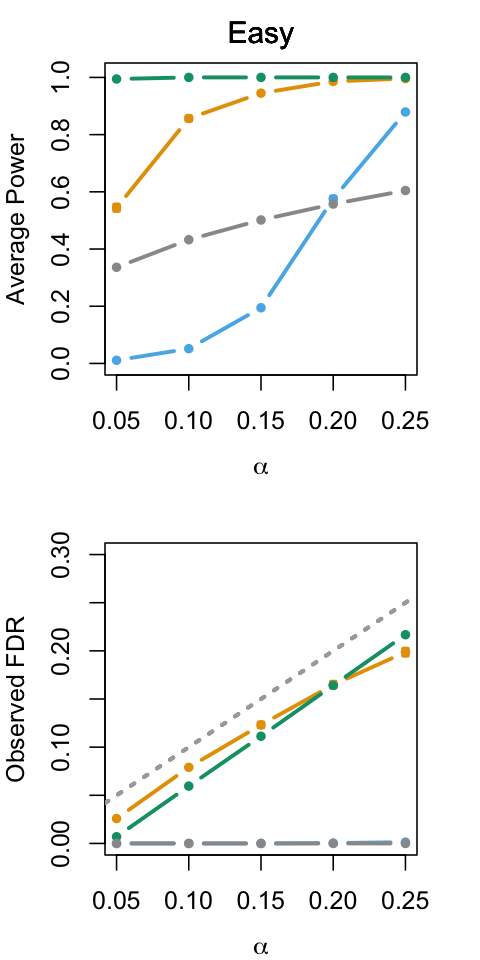} 
  \includegraphics[width=0.3\textwidth]{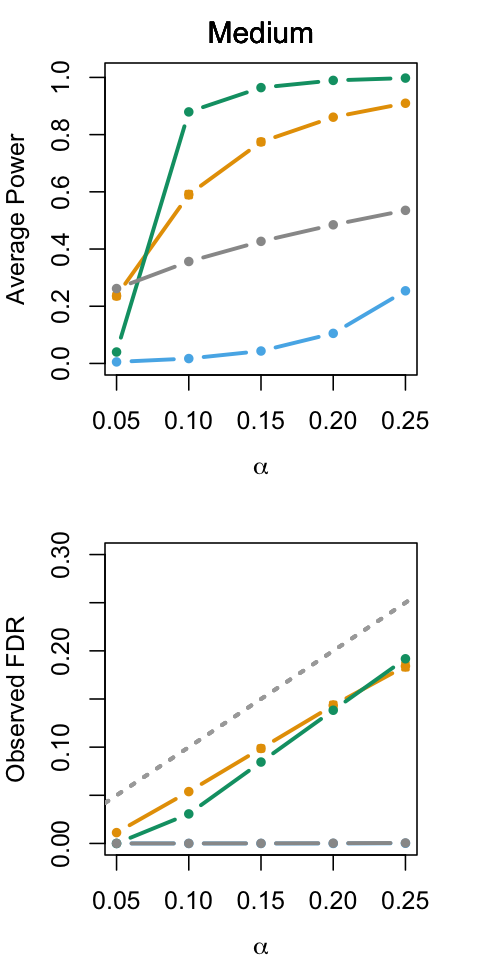}
  \includegraphics[width=0.3\textwidth]{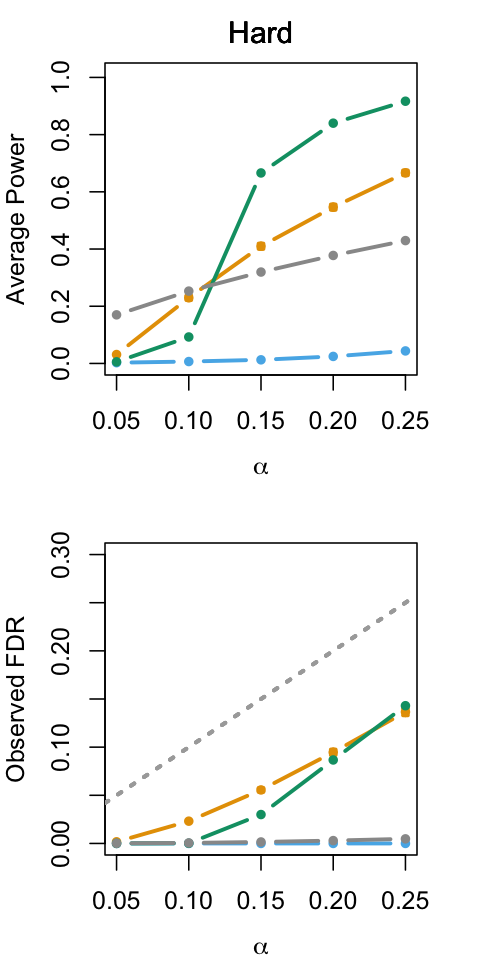}
  \caption{Effect of increasing the total number of hypotheses.  Instead of $100$ hypotheses of which $20$ are non-null, we consider $1000$ hypotheses of which $200$ are non-null.  With the exception of $\alpha$-thresholding, the performance of the methods remains largely unchanged.  One small change is that \emph{ForwardStop} loses power around $\alpha=0.1$ in the Hard setting.   The key difference is that the performance of $\alpha$-thresholding considerably degrades.  This is not surprising when we consider that $\alpha$-thresholding is simply a geometric random variable.  Thus as we increase the number of non-null hypotheses we expect the average power of $\alpha$-thresholding to drop to $0$.  }
  \label{fig:cp_appendix_diffn}
\end{figure}

\begin{figure}[p]
  \centering
  \includegraphics[width=0.75\textwidth]{./figures/main_legend.pdf} \\
  \includegraphics[width=0.3\textwidth]{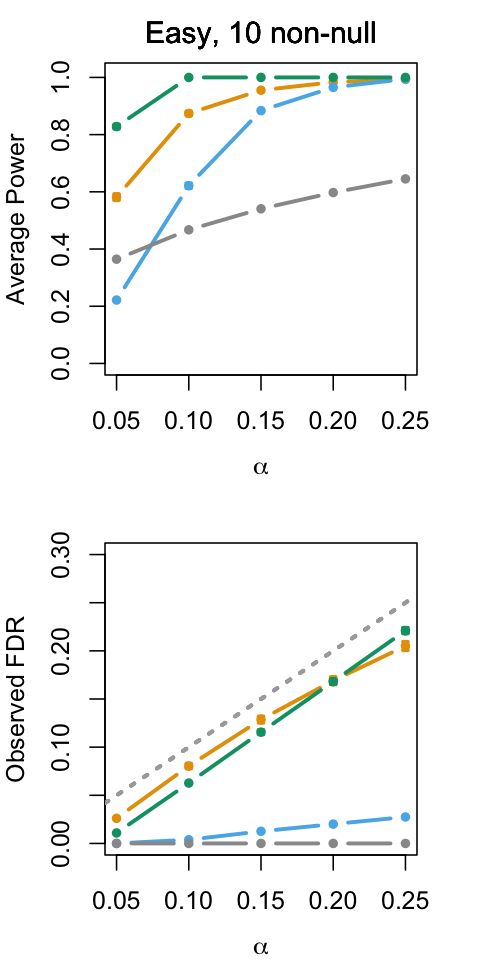} 
  \includegraphics[width=0.3\textwidth]{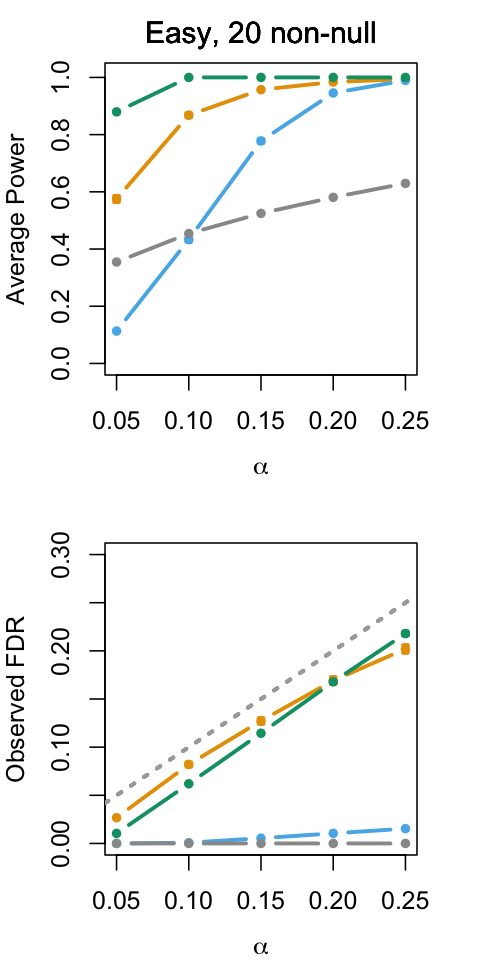}
  \includegraphics[width=0.3\textwidth]{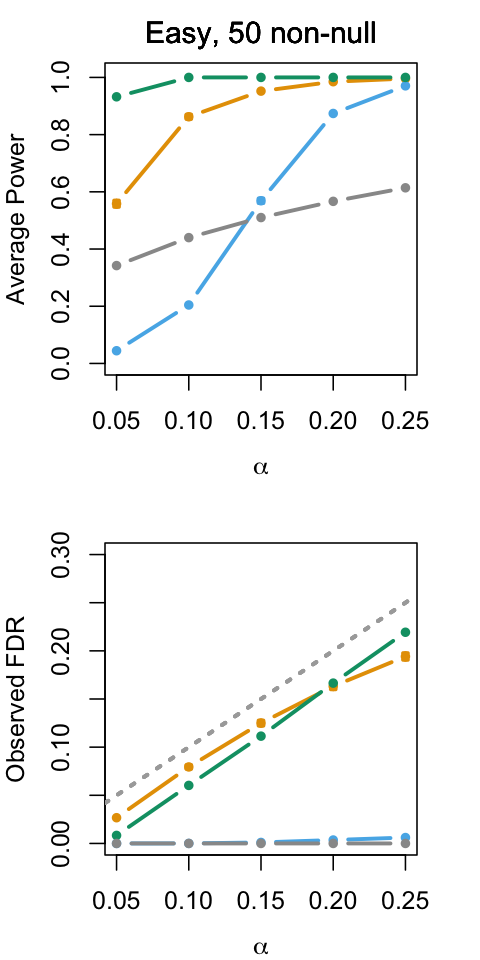}
  \caption{Effect of varying the number of non-nulls out of $m=100$ total hypotheses: Easy regime.  With the exception of $\alpha$-thresholding, the performance of the methods remains largely unchanged.  The performance of $\alpha$-thresholding degrades considerably as the number of non-null hypotheses increases.  An explanation for this behaviour is presented in \ref{fig:cp_appendix_diffn}.}
  \label{fig:cp_appendix_perfect_pi0}                        
\end{figure}                                                 
                                                             
\begin{figure}[p]                                          
  \centering                                                 
  \includegraphics[width=0.3\textwidth]{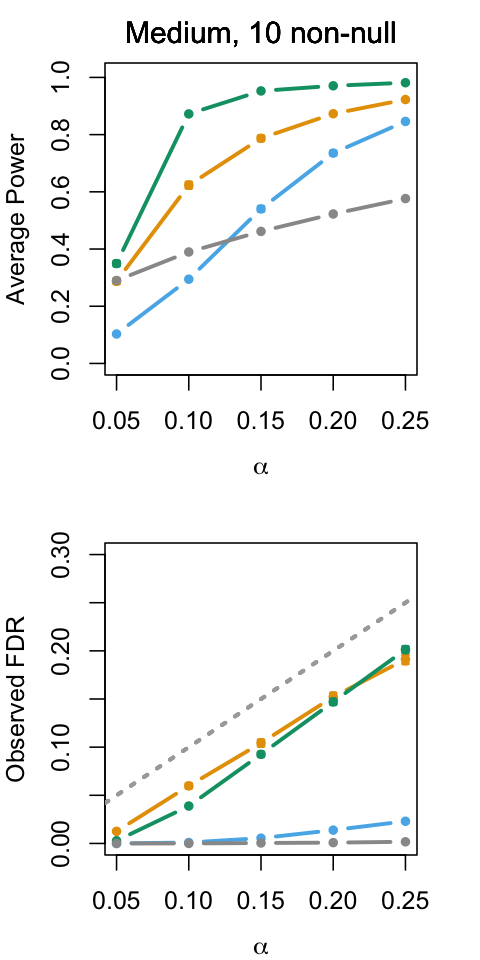} 
  \includegraphics[width=0.3\textwidth]{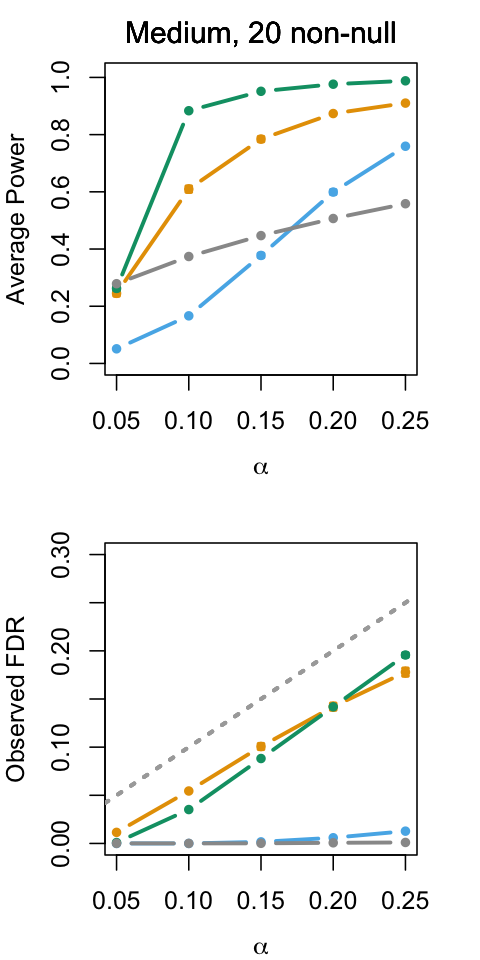}
  \includegraphics[width=0.3\textwidth]{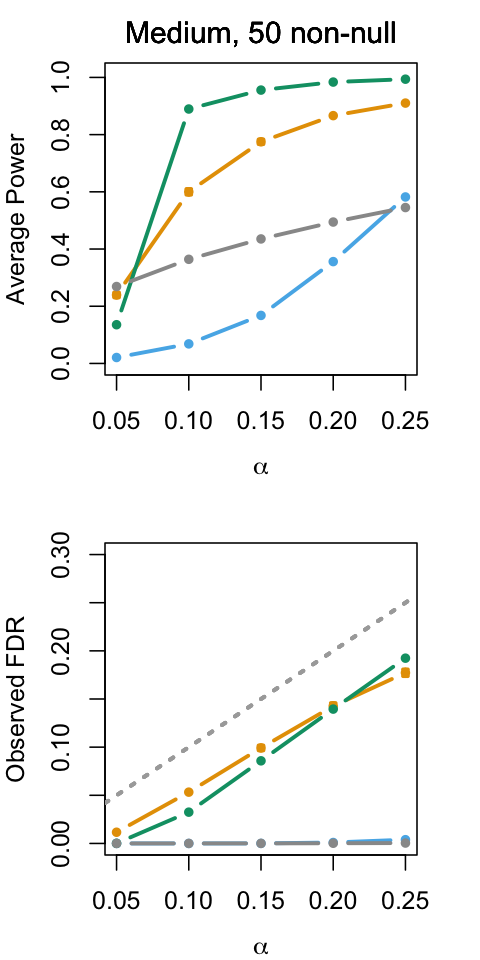}
  \caption{Effect of varying the number of non-nulls out of $m=100$ total hypotheses: Medium regime. The effect of varying the number of non-null hypotheses is qualitatively the same as in the Easy regime.}
  \label{fig:cp_appendix_good_pi0}                           
\end{figure}                                                 
                                                             
\begin{figure}[p]                                          
  \centering               
  \includegraphics[width=0.3\textwidth]{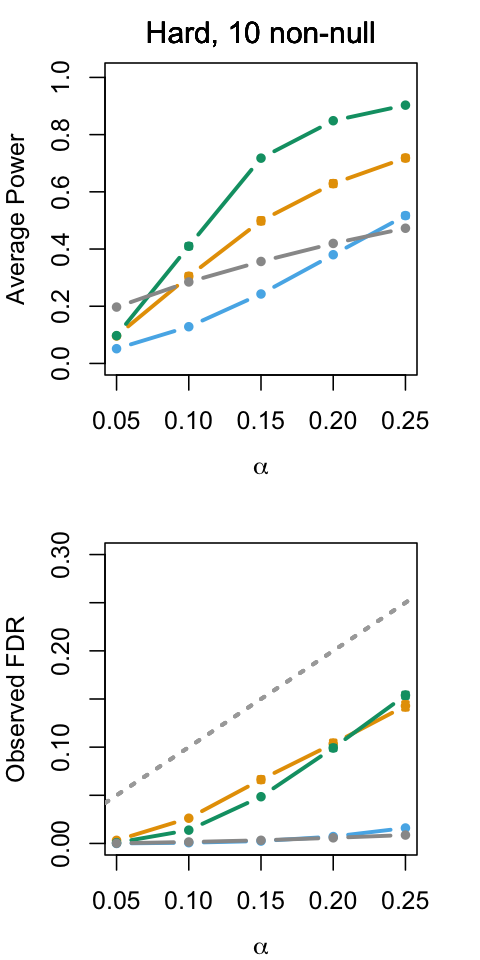} 
  \includegraphics[width=0.3\textwidth]{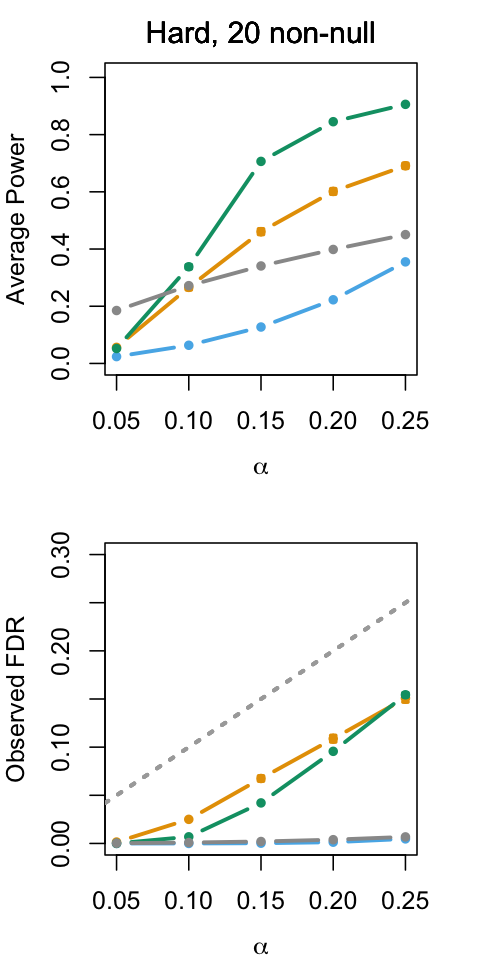}
  \includegraphics[width=0.3\textwidth]{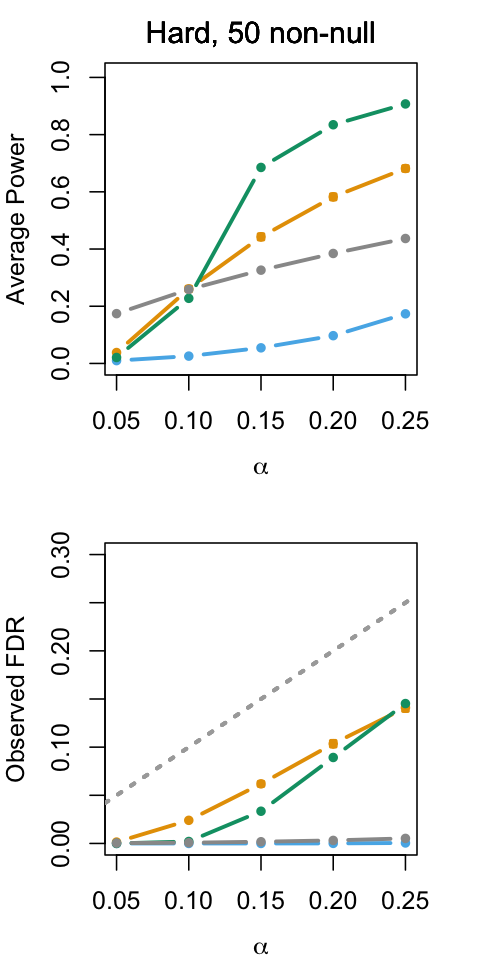}
  \caption{Effect of varying the number of non-nulls out of $m=100$ total hypotheses: Hard regime. The effect of signal strength is qualitatively the same as in the Easy and Medium regimes.}
  \label{fig:cp_appendix_moderate_pi0}
\end{figure}

\end{document}